\documentclass[11pt]{amsart}

\usepackage[T1]{fontenc}

\usepackage{amsmath,amsfonts,amssymb,amsthm, mathtools}
\usepackage[all]{xy}
\usepackage{csquotes}
\usepackage[usenames,dvipsnames]{xcolor}
\usepackage{tikz}

\usepackage{hyperref}
\hypersetup{colorlinks,  citecolor=blue,  linkcolor=BrickRed, urlcolor=black}

\usetikzlibrary{decorations,decorations.pathreplacing,decorations.markings}

\tikzstyle{mid>}=[decoration={markings, mark=at position 0.5 with {\arrow{>}}}, postaction={decorate}]
\tikzstyle{mid<}=[decoration={markings, mark=at position 0.5 with {\arrow{<}}}, postaction={decorate}]

\newcommand\bX{{\mathbb{X}}}
\newcommand\bY{{\mathbb{Y}}}
\newcommand\oY{{\overline{Y}}}
\newcommand\tY{\widetilde{Y}}
\newcommand\Gr{{\sf{Gr}}}
\newcommand\G{{\mathbb{G}}}
\DeclareMathOperator\Spec{Spec}
\renewcommand\P{{\mathbb{P}}}   
\newcommand\bA{{\mathbb{A}}}

\newcommand\Dqcoh[1]{D_{qcoh}(#1)}
\newcommand\sE{{\mathcal{E}}}
\newcommand\sF{{\mathcal{F}}}
\newcommand\sG{{\mathcal{G}}}
\newcommand\sL{{\mathcal{L}}}
\newcommand\sM{{\mathcal{M}}}
\renewcommand\O{\mathcal{O}}    
\newcommand\sP{{\mathcal{P}}}   
\newcommand\sQ{{\mathcal{Q}}}
\newcommand\sT{{\mathcal{T}}}
\newcommand\sV{{\mathcal{V}}}
\newcommand\D{\mathbb D}

\newcommand{\RG}{\mathcal{R}ep(G)}
\DeclareMathOperator\colim{colim}
\newcommand\Cone{{\operatorname{Cone}}}
\newcommand\sC{{\mathcal{C}}}
\newcommand\sD{{\mathcal{D}}}
\newcommand\tsC{{\tilde{\mathcal{C}}}}
\newcommand\tsD{{\tilde{\mathcal{D}}}}
\DeclareMathOperator{\id}{id}
\DeclareMathOperator{\Hom}{Hom}
\DeclareMathOperator{\End}{End}
\DeclareMathOperator{\Ext}{Ext}
\newcommand\Kom{{\sf{Kom}}}
\newcommand\sH{{\mathcal{H}}}

\renewcommand\1{{{\bf 1}}}      
\newcommand\A{{\sf{A}}}
\newcommand\E{{\sf{E}}}
\newcommand\F{{\sf{F}}}
\newcommand\R{{\sf{R}}}
\newcommand\T{{\sf{T}}}
\newcommand\X{{\sf{X}}}
\newcommand\Y{{\sf{Y}}}
\newcommand\dsss\sqsubseteq

\newcommand{\K}{\mathcal{K}}
\newcommand{\tK}{\widetilde{\mathcal{K}}}
\newcommand\uk{{\underline{k}}}

\newcommand\spn{{\rm{span}}}
\newcommand\Sym{{\rm{Sym}}}
\newcommand\tr{{\rm{tr}}}
\newcommand\ch{{\rm{ch}}}

\newcommand\bi{{\mathbf{i}}}
\newcommand\gen{{\mathrm{gen}}}

\newcommand\tpi{{\tilde{\pi}}}
\newcommand{\tg}{\widetilde{\mathfrak{g}}}
\newcommand{\g}{\mathfrak{g}}
\newcommand{\hI}{\widehat{I}}

\renewcommand{\sl}{\mathfrak{sl}}
\newcommand{\Lsl}{L\mathfrak{sl}}

\newcommand{\gl}{\mathfrak{gl}}
\newcommand\hW{\widehat{W}}
\newcommand\tsN{\widetilde{{\mathcal N}}}
\newcommand\sN{{\mathcal N}}
\newcommand\sR{{\mathcal R}}    

\newcommand{\N}{\mathbb{N}}
\newcommand{\bZ}{\mathbb Z}
\newcommand{\C}{\mathbb{C}}
\newcommand{\bR}{\mathbb{R}}
\newcommand\la{\langle}
\newcommand\ra{\rangle}
\newcommand\ttimes{\tilde{\times}}
\renewcommand\l{\lambda}        
\newcommand\vareps{\varepsilon}
\renewcommand\t{{\theta}}       
\newcommand\ux{{\underline{x}}}

\newcommand\ul{{\underline{\lambda}}}

\theoremstyle{plain}
\newtheorem{Theorem}{Theorem}[section]
\newtheorem{Proposition}[Theorem]{Proposition}
\newtheorem{Lemma}[Theorem]{Lemma}
\newtheorem{Corollary}[Theorem]{Corollary}

\theoremstyle{definition}

\theoremstyle{remark}
\newtheorem{Remark}[Theorem]{Remark}

\begin{document}

\setcounter{tocdepth}{1}

\title{Exotic t-structures and actions of quantum affine algebras}

\author{Sabin Cautis}
\email{cautis@math.ubc.ca}
\address{Department of Mathematics\\ University of British Columbia \\ Vancouver BC, Canada}

\author{Clemens Koppensteiner}
\email{clemens@ias.edu}
\address{School of Mathematics\\ Institute for Advanced Study \\ Princeton, NJ, USA}

\begin{abstract}
We explain how quantum affine algebra actions can be used to systematically construct ``exotic'' t-structures. The main idea, roughly speaking, is to take advantage of the two different descriptions of quantum affine algebras, the Drinfeld--Jimbo and the Kac--Moody realizations. 

Our main application is to obtain exotic t-structures on certain convolution varieties defined using the Beilinson--Drinfeld and affine Grassmannians. These varieties play an important role in the geometric Langlands program, knot homology constructions, K-theoretic geometric Satake and the coherent Satake category. As a special case we also recover the exotic t-structures of Bezrukavnikov--Mirkovi\'c \cite{BM} on the (Grothendieck--)Springer resolution in type A. 

\end{abstract}


\maketitle

\tableofcontents

\section{Introduction}

Consider a semisimple complex group $G$ and let $\RG$ denote the usual monoidal category of finite-dimensional representations of $G$. Consider further the Langlands dual group $G^\vee$ and its affine Grassmannian $\Gr = G^\vee((t))/G^\vee[[z]]$. The abelian category $P(\Gr)$ of $G^\vee[[z]]$-equivariant perverse sheaves on $\Gr$ is monoidal with respect to the convolution product. The geometric Satake correspondence of Mirkovi\'c--Vilonen \cite{MVi} states that there is an equivalence of monoidal categories
$$\RG \cong P(\Gr).$$
Subsequently $P(\Gr)$ is called the Satake category. 

There is a parallel story one can try to pursue on the coherent sheaf side -- namely if one replaces constructible sheaves with coherent sheaves. The analogue of perverse sheaves in then the notion of perverse coherent sheaves, in the sense of \cite{AB}. This notion is in general not as well behaved. 

One of the obvious issues is that over $\bR$ (i.e. on the constructible sheaf side) one can always use the middle perversity function $p(x) = - \frac{1}{2} \dim_{\bR}(x)$. On the coherent sheaf side this choice is not always possible since $\dim_\C(x)$ may not always be even. However, when it is, choosing $p(x) = -\frac{1}{2}\dim_\C(x)$ leads to a nice theory theory which includes, for instance, the notion of intermediate extension \cite[Section 4]{AB}. 

In the case of $\Gr$ the $G^\vee[[z]]$-orbits have the same parity on any connected component. This allows us to use the middle perversity function to obtain a well behaved abelian category $PCoh(\Gr)$. Moreover, as shown in \cite[Section 8]{BFM}, the convolution product on $\Gr$ preserves $PCoh(\Gr)$. Hence, $PCoh(\Gr)$ is monoidal and we call this the coherent Satake category. 

The coherent Satake category $PCoh(\Gr)$ behaves quite differently from $P(\Gr)$ and is interesting in its own way. For example, it is closely tied to $4d$ ${\mathcal N} = 2$ gauge theory. Although not semisimple it carries the structure of a monoidal cluster category \cite{CW}. In order to understand $PCoh(\Gr)$ one needs to understand the convolution product and subsequently the convolution varieties, as we now explain.

\subsection{Convolution varieties} \label{sec:conv-vars}

The $G^\vee[[z]]$-orbits of $\Gr$ are indexed by dominant weights of $G$. For such a weight $\l$ we denote by $\Gr^\l$ the closure of the corresponding orbit in $\Gr$. For a sequence $\ul = (\l_1, \dots, \l_n)$ of such weights we can consider the convolution product variety 
$$\Gr^{\ul} = \Gr^{\l_1} \ttimes \cdots \ttimes \Gr^{\l_n}.$$
This variety comes equipped with a natural map to $\Gr$ whose image is $\Gr^{\sum_i \l_i}$. The resulting map $m: \Gr^{\ul} \rightarrow \Gr^{\sum_i \l_i}$, which is semi-small, is used to define the convolution product on $PCoh(\Gr)$. 

{\bf Example.} To illustrate the type of geometry involved let us consider the simplest nontrivial case. We take $G=SL_2$ and $\l_1=\l_2=\omega$ where $\omega$ is the fundamental weight. Then $\Gr^{(\omega,\omega)}$ is isomorphic to the natural bundle compactification of $T^\star \P^1$ (the cotagent bundle of $\P^1$) while $\Gr^{2 \omega}$ is isomorphic to the quadric cone in $\P^3$. The map $m: \Gr^{(\omega,\omega)} \rightarrow \Gr^{2 \omega}$ is then the projection which collapses the zero section (i.e. the $-2$ curve) in $T^\star \P^1$ to a point. Note that $\Gr^{(\omega,\omega)}$ is smooth while $\Gr^{2 \omega}$ has one singular point. 

These convolution spaces are interesting in themselves. For instance, in \cite{CK1} a plan was laid out for using these varieties to categorify the Reshetikhin--Turaev link invariants. This program was subsequently carried out for $G=SL_m$ in a series of papers. Meanwhile in \cite{CK5} these spaces are used to suggest a (quantum) K-theoretic analogue of the geometric Satake equivalence.

One of the goals of this paper is to study natural t-structures on the derived category of coherent sheaves of these convolution varieties. As mentioned above, the varieties $\Gr^\l$ carry a natural t-structure corresponding to perverse coherent sheaves. We would like to find a ``lift'' of these t-structures to $\Gr^{\ul}$. Since the dimensions of the $G^\vee[[z]]$-orbits in $\Gr^{\ul}$ do not have the same parity, the category of perverse coherent sheaves on $\Gr^{\ul}$ is not a good choice. 

In this paper we will achieve this goal in the case $G=SL_m$ and $\ul = (\omega_{k_1}, \dots, \omega_{k_n})$ where $\omega_1, \dots, \omega_{m-1}$ are the fundamental weights of $SL_m$. Following our notation from earlier papers we will write $\uk = (k_1, \dots, k_n)$ and denote $\Gr^{\ul}$ in this case as $Y(\uk)$. One of our main results, Theorem \ref{thm:skew-app}, defines and characterizes a certain family of natural t-structure on $Y(\uk)$.  

The argument behind Theorem \ref{thm:skew-app} involves a general approach to constructing t-structures using categorical actions of quantum affine algebras. Before explaining this in more detail (see Section \ref{sec:qaffaction}) we describe further motivation for this paper coming from an approach to proving some conjectures by Lusztig. 

\subsection{Lusztig's conjectures and the work of Bezrukavnikov, Mirkovi\'c and Rumynin}

Let $G$ be a semisimple group with Lie algebra $\g$. Lusztig \cite{L} conjectured the existence and properties of certain \enquote{canonical} bases in the Grothendieck group of Springer fibers. These bases can be used to derive precise information about numerical invariants of $\g$-modules. Recently, Bezrukavnikov and Mirkovi\'c \cite{BM} proved most of these conjectures. A key part of their argument is the construction of two ``exotic'' t-structures on the Springer resolution $\pi\colon \tsN \rightarrow \sN$ of the nilpotent cone $\sN$ (as well as on the Grothendieck--Springer resolution $\tpi\colon \tg \rightarrow \g$). 

The first t-structure is called the ``representation theoretic (RT) exotic t-structure''. It is defined by a locally free tilting bundle $\sE$ on $\tg$, whose pullback to $\tsN$ is also tilting. The bundle $\sE$ is constructed by repeatedly applying certain reflection functors $\sR_i$ for $i \in \hI$ (the affine Dynkin diagram of $\g$) to the structure sheaf $\O_{\tg}$. The most difficult part of this argument is proving the necessary vanishing condition $\Ext^{>0}(\sE,\sE) = 0$ (see \cite[Section 2.4]{BM}). 

This is done in two steps. First one reduces to the corresponding vanishing on the formal neighborhood of the zero section in $\tg$ \cite[Section 2.5.1]{BM}. Then one uses the derived equivalence, studied in \cite{BMR}, between the category of coherent sheaves on this formal neighbourhood and the derived category of $\g$-modules with a fixed generalized central character. This requires passing to positive characteristic and then lifting back up. 

The second t-structure, called the ``perversely exotic t-structure'', is defined using the same tilting bundle $\sE$ together with the category of perverse coherent sheaves, in the sense of \cite{AB}, on $\sN$. This t-structure can also be described as follows. From \cite{ArkB} there exists an equivalence 
\begin{equation}\label{eq:Phi}
\Phi\colon D^G(\tsN) \xrightarrow{\sim} D(Perv_{Fl})
\end{equation}
between the derived category of $G$-equivariant coherent sheaves on $\tsN$ and the derived category of anti-spherical perverse sheaves on the affine flag variety $Fl$ of the dual group. Under this equivalence the perversely exotic t-structure on $D^G(\tsN)$ corresponds to the perverse t-structure on $D(Perv_{Fl})$ \cite[Lemma~6.2.4]{BM}. The equivalence $\Phi$ is ultimately used to reduce a key step in the proof of Lusztig's conjectures to a purity statement in $D(Perv_{Fl})$.

The perversely exotic t-structure was already studied in \cite{B}. That paper considered a functor 
$$\Psi\colon D^{G \times \C^\times}(\tsN) \rightarrow D({\sf{U}}_q\text{--}mod^0)$$
(which is almost an equivalence) and showed that irreducible perverse exotic sheaves on $D^{G \times \C^\times}(\tsN)$ correspond to indecomposable tilting objects in ${\sf U}_q\text{--}mod^0$. Here ${\sf U}_q$ is Lusztig's quantum enveloping algebra at a primitive root of unity $q \in \C$ and ${\sf U}_q\text{--}mod^0$ is the block containing the trivial module inside the category ${\sf U}_q\text{--}mod$ of finite dimensional graded ${\sf U}_q$-modules. 

The current paper is also motivated by our wish to better understand these t-structures and to find an alternative construction (see Section \ref{sec:finalrem}). In particular, we wanted to avoid having to pass to positive characteristic or having to use a deep result such as the derived equivalence from \cite{BMR}. This is because in other cases, such as our convolution varieties $Y(\uk)$, no such equivalences are known (or expected to exist). 

\subsection{Quantum affine algebra actions}\label{sec:qaffaction}

We refer to Section \ref{sec:action} for the precise definition of a categorical $U_q(L\gl_n)$ action. One of the main results of this paper, Theorems \ref{thm:sym:main} and \ref{thm:skew:main}, explains how one can use an action of $U_q(L\gl_n)$ to construct ``exotic'' t-structures. More precisely, suppose we start start with an action of $U_q(L\gl_n)$ on a graded triangulated category 
$$\tK = \ \ \bigoplus_{\mathclap{\uk = (k_1,\dots,k_n)}} \ \ \tK(\uk)$$ 
where multiplication of $q$ acts as an ``internal'' shift $\{1\}$ (rather than a cohomological shift $[1]$). Then, having fixed a t-structure on the highest weight category $\tK(\eta)$ where $\eta = (0,\dots,0,N)$, Theorem \ref{thm:sym:main} says that under some mild assumption on the t-structure on $\tK(\eta)$, there exists a unique t-structure on $\oplus_{\uk} \tK(\uk)$ such that
\begin{enumerate}
\item it matches the chosen t-structure on $\tK(\eta)$ and
\item the generators $\{\E_i,\F_i: i \in \hI\}$ of $U_q(L\gl_n)$ act by t-exact functors.
\end{enumerate} 
Then Theorem \ref{thm:skew:main} explains what to do if we have an action on a graded triangulated category $\K = \oplus_{\uk} \K(\uk)$ where $q$ acts as a {\it{cohomological}} shift. 

We can apply these results to our convolution varieties as follows. First we consider certain varieties $\bY(\uk)$, which are analogous to $Y(\uk)$, but occur in the geometry of the Beilinson--Drinfeld Grassmannian rather than the affine Grassmannian. In \cite{CK4} an action of $U_q(L\gl_n)$ on $\oplus_\uk D(\bY(\uk))$ is constructed where $q$ acts as a $\C^\times$-equivariance shift $\{1\}$. This allows us to apply Theorem \ref{thm:sym:main} with $\tK(\uk) := D(\bY(\uk))$ to obtain exotic t-structures on $\oplus_{\uk} D(\bY(\uk))$ (cf. Theorem \ref{thm:sym-app}). 

Then we take $\K(\uk) := D(Y(\uk))$. In the papers \cite{CKL1,CKL2,C2} one defined an action of $U_q(\gl_n)$ on $\K(\uk)$, which we extend to an action of $U_q(L\gl_n)$ in Section \ref{sec:cattk}. We then apply Theorem \ref{thm:skew:main} to obtain exotic t-structures on $\oplus_{\uk} D(Y(\uk))$ (cf. Theorem \ref{thm:skew-app}). 

As a consequence of these results one immediately recovers the RT exotic t-structure on $D(\tg)$ (Corollary \ref{cor:tg-app}) as well as the RT and perversely exotic t-structures on $D(\tsN)$ (Corollary \ref{cor:g-app}) in the case when $\g = \gl_m$. 

\subsection{Drinfeld--Jimbo vs.\ Kac--Moody}
The quantum group $U_q(L \gl_n)$ has two realizations: the Drinfeld--Jimbo (DJ) realization (sometimes refered to as the loop realization) and the Kac--Moody (KM) realization. The DJ realization is generated by 
$$\{E_i \otimes t^j, F_i \otimes t^j: i \in I \text{ and } j \in \bZ \}$$ 
whereas the KM one is generated by $E_i,F_i$ where $i \in \hI$ (the vertex set of the affine Dynkin diagram). When it is important to distinguish between these two we will write $U_q(L\gl_n)_{DJ}$ or $U_q(L\gl_n)_{KM}$. 

At the categorical level there is also a distinction between these two. In a categorical $U_q(L\gl_n)_{KM}$ the relations between functors are given by isomorphisms. For example, an important one is the ``$\sl_2$ relation''
$$\E_i \F_i \1_\l \cong \F_i \E_i \1_\l \oplus_{[\la \l,\alpha_i \ra]} \1_\l$$
where $\1_\l$ denotes the identity functor on the weight space indexed by $\l$. In particular, it makes sense for $U_q(L\gl_n)_{KM}$ to act on additive (graded) categories. Such actions have been studied in much detail \cite{CR,KL,Ro,C3}. 

On the other hand, the definition of a categorical action of $U_q(L\gl_n)_{DJ}$ is less clear (see \cite{CK4} for a working definition). What is apparent however, is that with such an action one needs some sort of triangulated structure to express all the relations between functors. For example, if $\la \alpha_i, \alpha_j \ra = -1$, then one requires an isomorphism of cones 
$$\Cone(\E_i \E_{j,1} \la -1 \ra \xrightarrow{\alpha} \E_{j,1} \E_i) \cong \Cone(\E_j \E_{i,1} \la -1 \ra \xrightarrow{\beta} \E_{i,1} \E_j)$$
where $\alpha$ and $\beta$ are the unique nonzero maps (here $\E_{i,1}, \E_{j,1}$ are the functors corresponding to $E_i \otimes t, E_j \otimes t$). In particular, it makes sense that $U_q(L\gl_n)_{DJ}$ should act on triangulated (graded) categories. 

To pass between the KM and DJ realizations one writes $E_0$ and $F_0$ as a linear combinations of compositions of $E_i \otimes t^j$ and $F_i \otimes t^j$. The analogous result at the categorical level of this is explained in Appendix \ref{sec:appdrinfeld}. It turns out that $\E_0,\F_0$ are given by complexes of functors corresponding to $E_i \otimes t^j, F_i \otimes t^j$. This suggests that, at the categorical level, the equivalence of these two realizations is achieved as an equivalence between homotopy categories: roughly speaking an equivalence between an $U_q(L\gl_n)_{DJ}$ action and the homotopy category of a $U_q(L\gl_n)_{KM}$ action. 

Now consider an action of $U_q(L\gl_n)_{DJ}$ on a triangulated (graded) category $\K$. Switching to the Kac--Moody realization one obtains an action of $U_q(L\gl_n)_{KM}$ on $\K$. In ideal circumstances, the triangulated category $\K$ would be the homotopy category of an additive category on which $U_q(L\gl_n)_{KM}$ acts. A weaker expectation is that $\K$ should be equipped with a t-structure where the generators of $U_q(L\gl_n)_{KM}$ is t-exact. The purpose of this paper is to make this expectation precise. 

Our main applications are geometric (Sections \ref{sec:examples} and \ref{sec:t}). This is not surprising since, going back to work of Nakajima \cite{N} on quiver varieties, actions of $U_q(\g)$ on categories of coherent sheaves often extend to an action of $U_q(L\g)$ actions -- the extra loop structure usually coming by way of line bundles. When $\g$ is of finite type it makes sense to expect that the generators of $U_q(L\g)_{KM}$ are exact with respect to some exotic t-structure.

\subsection{Further remarks}

The approach in this paper should give an explicit way to understand the irreducible objects in the heart of these t-structures. We expect that these are given as the image of the irreducible objects in the highest weight category under affine forest-like maps. Such maps are morphisms in $U_q(L\gl_n)_{KM}$ associated to a collection of affine trees (trees are discussed in Section \ref{sec:trees}). In many cases the highest weight category is the category of graded vector spaces so there is only one irreducible (up to grading shift). This not only gives a combinatorial way to identify the irreducibles but also an algebraic/algorithmic way to compute the $\Ext$'s between them -- namely these $\Ext$-spaces should be expressible in terms of quiver Hecke algebras (which act by natural transformations in a categorical action of $U_q(L\gl_n)_{KM}$). 

One case where such irreducibles have been studied in more detail is that of two-block Springer fibers \cite{AN}. In that paper, Anno and Nandakumar show that irreducibles are indexed in a natural way by affine crossingless matchings. In our language, these affine crossingless matchings correspond to affine forests where each strand is labeled $1$. More precisely, strands are generally labeled from $\{1,\dots,m-1\}$ for some $m \in \N$, but the fact that we are dealing with two-block Springer fibers implies that $m=2$. When $m=2$ affine forests are the same as affine crossingless matchings.  

In order to directly extend our construction of exotic t-structures to arbitrary $\Gr^{\ul}$ we need an affine braid group action on such varieties. Although such an action is not known in general it is possible to define (building on some recent results of the first author) in the special case $\ul = (k_1 \omega_1, \dots, k_n \omega_1)$. This will allow us to also define exotic t-structures for these spaces in an upcoming paper. 

One of the interesting results from \cite{BM} is Property ($\star$) from Section 5.3.2. This is a vanishing statement for certain $\Ext$s which involves the internal grading $\{\cdot\}$ on $D^G(\tsN)$. This grading, which is induced by the natural $\C^\times$ action on $\tsN$, corresponds under the equivalence $\Phi$ from \eqref{eq:Phi} to the grading by Frobenius weights. Bezrukavnikov and Mirkovi\'c then apply the Purity Theorem of \cite{BBD} to obtain this vanishing. A similar vanishing result, namely the Positivity Lemma \cite[Lemma 9]{B}, is one of the main tools in that paper. The proof of this Lemma also uses \eqref{eq:Phi} and the Purity Theorem. It would be interesting to obtain these result directly (and in greater generality) from the action of $U_q(L\gl_n)$ without having to turn to the equivalence $\Phi$ or the Purity Theorem. This would involve defining weight structures (a.k.a.~co-t-structures) in the sense of \cite{Bon,Pa}. 

Finally, perverse exotic sheaves on the convolution varieties $Y(\uk)$ should be thought of analogues of perverse coherent sheaves on the affine Grassmannian (the coherent Satake category). Although there are some concrete statements one can make (for example, the pushforward of a perverse exotic sheaf to the affine Grassmannian is perverse coherent) the relationship of perverse exotic sheaves to the coherent Satake category is something that begs further study. 

\subsection*{Acknowledgements} We were supported by NSERC through a discovery/ac\-cel\-er\-a\-tor grant. S.C.~would like to thank Ivan Mirkovi\'c for sharing, back in 2009, a preliminary version of \cite{BM} which turned out to be inspiring (and a bit amusing).

\section{Preliminaries}

We will always work over the base field $\C$ -- although much (if not everything) of what we do should work in arbitrary characteristic. 

For $n \ge 1$ we denote by $[n]$ the quantum integer $q^{n-1} + q^{n-3} + \dots + q^{-n+3} + q^{-n+1}$. By convention $[-n] = -[n]$. If $f = \sum_a f_a q^a \in \N[q,q^{-1}]$ and $\A$ is an object inside a graded category $\K$ we write $\bigoplus_f \A$ for the direct sum $\bigoplus_{a \in \bZ} \A^{\oplus f_a} \la a \ra$. For example, 
$$\bigoplus_{[n]} \A = \bigoplus_{k=0}^{n-1} \A \la n-1-2k \ra.$$

\subsection{Categorical actions}\label{sec:action}
By a graded 2-category $\K$ we mean a 2-category whose 1-morphisms are equipped with an auto-equivalence $\la 1 \ra$. $\K$ is idempotent complete if for any 2-morphism $\alpha$ with $\alpha^2=\alpha$ the image of $\alpha$ is contained in $\K$. 

We now recall the notion of a $(L\gl_n,\theta)$ action based on \cite{C3}. We will work with the weight lattice of $\gl_n$, which we can identify with $\bZ^n$ equipped with the standard bilinear form $\la \cdot,\cdot \ra: \bZ^n \times \bZ^n \rightarrow \bZ$. Elements of $\bZ^n$ are denoted $\uk = (k_1, \dots, k_n)$. The affine root sublattice in this case is generated by $\alpha_i = (0, \dots, -1,1,\dots,0)$ for $i \in I := \{1,\dots,n-1\}$ where the $-1$ appears in position $i$. We write $\alpha_0 = (1,0,\dots,0,-1)$ and $\hI := I \cup \{0\}$. We also shorten $\la \alpha_i, \alpha_j \ra$ as $\la i,j \ra$ for $i,j \in \hI$. 

The type of $(L\gl_n,\t)$ actions we consider consist of a target graded, additive, $\C$-linear idempotent complete 2-category $\K$ where the objects $\K(\uk)$ are indexed by $\uk \in \bZ^n$ and equipped with:
\begin{enumerate}
\item 1-morphisms: $\E_i \1_\uk = \1_{\uk+\alpha_i} \E_i$ and $\F_i \1_{\uk+\alpha_i} = \1_\uk \F_i$ where $i \in \hI$ and $\1_\uk$ is the identity 1-morphism of $\K(\uk)$, 
\item 2-morphisms: for each $\uk \in \bZ^n$, a linear map $\spn\{\alpha_i: i \in \hI \} \rightarrow \End^2(\1_\uk)$.
\end{enumerate}

On this data we impose the following conditions.

\begin{enumerate}
\item \label{co:hom1} $\Hom(\1_\uk, \1_\uk \la l \ra)$ is zero if $l < 0$ and one-dimensional if $l=0$ and $\1_\uk \ne 0$. Moreover, the space of maps between any two 1-morphisms is finite dimensional.
\item \label{co:adj} $\E_i$ and $\F_i$ are left and right adjoints of each other up to specified shifts. More precisely
    \begin{enumerate}
    \item $(\E_i \1_\uk)^R \cong \1_\uk \F_i \la \la \uk, \alpha_i \ra + 1 \ra$
    \item $(\E_i \1_\uk)^L \cong \1_\uk \F_i \la -  \la \uk, \alpha_i \ra  -1 \ra$.
    \end{enumerate}

\item \label{co:EF} We have
    \begin{align*}
    \E_i \F_i \1_\uk &\cong \F_i \E_i \1_\uk \bigoplus_{[ \la \uk, \alpha_i \ra ]} \1_\uk \ \ \text{ if }  \la \uk, \alpha_i \ra  \ge 0 \\
    \F_i \E_i \1_\uk &\cong \E_i \F_i \1_\uk \bigoplus_{[- \la \uk, \alpha_i \ra ]} \1_\uk \ \ \text{ if }  \la \uk, \alpha_i \ra  \le 0
    \end{align*}

\item \label{co:EiFj} If $i \ne j \in I$ then $\F_j \E_i \1_\uk \cong \E_i \F_j \1_\uk$.

\item \label{co:theta} 
    The composition $\E_i \E_i$ decomposes as $\E_i^{(2)} \la -1 \ra \oplus \E_i^{(2)} \la 1 \ra$ for some 1-morphism $\E_i^{(2)}$. Moreover, if $\theta \in \spn \{\alpha_i: i \in \hI\}$ where $\la \theta, \alpha_i \ra \ne 0$ (resp. $\la \theta, \alpha_i \ra = 0$) then $I \t I \in \End^2(\E_i \1_\uk \E_i)$ induces a nonzero map (resp.\ the zero map) between the summands $\E_i^{(2)} \la 1 \ra$ on either side.


\item \label{co:vanish1} If $\alpha = \alpha_i$ or $\alpha = \alpha_i + \alpha_j$ for some $i,j \in \hI$ with $\la i,j \ra = -1$ then $\1_{\uk+r \alpha} = 0$ for $r \gg 0$ or $r \ll 0$.

\item \label{co:new} Suppose $i \ne j \in \hI$. If $\1_{\uk+\alpha_i}$ and $\1_{\uk+\alpha_j}$ are nonzero then $\1_\uk$ and $\1_{\uk+\alpha_i+\alpha_j}$ are also nonzero.
\end{enumerate}

In this paper we will always assume that $\K(\uk)$ is zero (i.e. $\1_\uk = 0$) if some $k_i < 0$. Thus condition (\ref{co:vanish1}) will be obvious. Condition (\ref{co:new}) will likewise be obvious. 

\begin{Remark}
Since $\K$ is graded a $(L\gl_n,\t)$ action induces an action of $U_q(L\gl_n)$ at the level of Grothendieck groups. We chose to write $(L\gl_n,\theta)$ rather than perhaps $(U_q(L\gl_n),\theta)$ in order to simplify notation. 

In the discussion above this action of $U_q(L\gl_n)$ appears in its Kac--Moody (KM) presentation. This in contrast to the action constructed in \cite{CK4} which appears in its Drinfeld--Jimbo (DJ) realization. When it is not clear from context and we need to differentiate between the two realizations we will write $(L\gl_n,\t)_{KM}$ or $(L\gl_n,\t)_{DJ}$ as the case may be (cf. Section \ref{sec:examples}). 
\end{Remark}

\subsection{Triangulated categories and convolution}\label{sec:convolution}

An additive 2-category $\K$, such as the ones from the last section, is said to be triangulated if for any two objects $\uk, \uk' \in \K$ the category $\Hom(\uk,\uk')$ is triangulated. For example, the homotopy 2-category $\Kom(\K)$ of any additive 2-category $\K$ is triangulated for the same reasons that the homotopy category of an additive category is triangulated. 

Consider now a triangulated category $\sC$ and a complex of objects 
$$(A_\bullet, d_\bullet) = [A_n \xrightarrow{d_n} \dotsb \xrightarrow{d_2} A_1 \xrightarrow{d_1} A_0]$$
with $A_i$ in cohomological degree $-i$. In the literature there are various conventions about how to define the convolution of this complex as an iterated cone. We will convolve from the right as follows.

The convolution of $(A_\bullet, d)$ is any object $B$ such that there exist
\begin{enumerate}
    \item objects $A_0 = B_0, B_1, \dotsc, B_n = B$ and
    \item morphism $f_i\colon A_{i}[-i] \to B_{i-1}$, $g_i\colon B_{i-1} \to B_{i}$ and $h_i\colon B_i \to A_i[-i+1]$ (with $h_0 = \id[1]$)
\end{enumerate}
such that
\[
    A_{i}[-i] \xrightarrow{f_i} B_{i-1} \xrightarrow{g_i} B_i \xrightarrow{h_{i}}
\]
is a distinguished triangle for each $i$ and $h_{i-1} \circ f_i = d_i$. Note that in general convolutions may not exist and they may not be unique. But under some reasonable conditions they both exist and are unique, cf.~\cite[Proposition~8.3]{CK1}.

As a particular case one finds that the convolution of a two term complex $[A \xrightarrow{f} B]$ with $B$ in degree zero is just $\Cone(f)$. On the other hand, if the $A_i$ belong to the heart of a t-structure then the convolution is unique and isomorphic (in the derived category) to the class of the complex 
$$A_n \xrightarrow{d_n} \dotsb \xrightarrow{d_2} A_1 \xrightarrow{d_1} A_0$$
where, as before, $A_0$ is in degree zero.

\subsection{t-structures}

Suppose $\sC$ and $\sD$ are triangulated categories. We say that a set of objects $\{X_s : s \in S \}$ {\it weakly generates} $\sC$ if $\Hom(X_s[n],Y)=0$ for all $s \in S$ and $n \in \bZ$ implies $Y=0$. A functor $\Phi\colon \sC \rightarrow \sD$ is {\it conservative} if $\Phi(X)=0$ implies $X=0$. 

Recall that a {\it t-structure} on $\sC$ consists of two full subcategories $\sC^{\le 0}$, $\sC^{\ge 0}$ such that:
\begin{itemize}
    \item $\sC^{\le -1} \subseteq \sC^{\le 0}$ and $\sC^{\ge 1} \subseteq \sC^{\ge 0}$, where $\sC^{\le n} \coloneqq \sC^{\le 0}[-n]$ and $\sC^{\ge n} \coloneqq \sC^{\ge 0}[-n]$;
    \item $\Hom(\sC^{\le 0}, \sC^{\ge 1}) = 0$;
    \item Every object $X$ can be embedded in a distinguished triangle $X^{\le 0} \to X \to X^{\ge 1}$ with $X^{\le 0}\in \sC^{\le 0}$, $X^{\ge 1}\in \sC^{\ge 1}$.
\end{itemize}
From these axioms it follows that the inclusion $\sC^{\le n} \hookrightarrow \sC$ has a right adjoint, denoted $\tau^{\le n}\colon \sC \to \sC^{\le n}$, while $\sC^{\ge n} \hookrightarrow \sC$ has a left adjoint denoted $\tau^{\ge n}\colon \sC \to \sC^{\ge n}$.
Moreover the heart of the t-structure, $\sC^\heartsuit = \sC^{\le 0} \cap \sC^{\ge 0}$ is an abelian category and the functor \[ H^n = \tau^{\le n} \circ \tau^{\ge n} = \tau^{\ge n} \circ \tau^{\le n}\colon \sC \to \sC^{\heartsuit} \] is a cohomological functor.

The following result of Polishchuk will be our main tool for defining t-structures. 

\begin{Theorem}[{\cite[Theorem~2.1.2]{Po}}] \label{thm:t-structure:right-exact-induces}
Let $\Phi \colon \tsC \to \tsD$ be a functor of cocomplete triangulated categories which commutes with all small coproducts and that admits a left adjoint functor $\Phi^L \colon \tsD \to \tsC$. Let $\sC \subseteq \tsC$ and $\sD \subseteq \tsD$ be full subcategories such that $\Phi(\sC) \subseteq \sD$, $\Phi^L(\sD) \subseteq \sC$ and $\Phi(X) \in \sD \Rightarrow X \in \sC$ for any $X \in \tsC$.
    
Assume that $\sD$ has a t-structure $(\sD^{\le 0},\sD^{\ge 0})$ such that $\Phi \circ \Phi^L \colon \sD \to  \sD$ is right t-exact.  Then there exists a (unique) t-structure on $\sC$ with
$$\sC^{\ge 0} = \bigl\{ X \in \sC : \Phi(X) \in \sD^{\ge 0} \bigr\}.$$ 
Moreover the functor $\Phi$ is t-exact with respect to these t-structures.
\end{Theorem}

From hereon we will always assume that whenever we want to apply Theorem \ref{thm:t-structure:right-exact-induces} to a functor of triangulated categories $\phi\colon \sC \rightarrow \sD$ we can extend it to a functor $\Phi\colon \tsC \rightarrow \tsD$ between cocomplete triangulated categories which satisfies all the conditions in the statement of Theorem \ref{thm:t-structure:right-exact-induces}.

Let us explain why we can do this. In all the examples we consider the categories $\sC$ and $\sD$ will be bounded derived categories $D(X)$ and $D(Y)$ of coherent sheaves on smooth varieties $X,Y$. We take $\tsC$ and $\tsD$ to be the corresponding (cocomplete) unbounded derived categories of quasi-coherent sheaves $\Dqcoh{X}$ and $\Dqcoh{Y}$. Any functor $\phi\colon \sC \to \sD$ will be given by a kernel and hence naturally extends to quasi-coherent sheaves. Furthermore, all such functors will have right adjoints and hence commute with small coproducts by the Adjoint Functor Theorem. Finally, $\phi$ will always be conservative and by Lemma~\ref{lem:conservative_and_compactness} below this means that $\Phi(X) \in \sD \Rightarrow X \in \sC$.

\begin{Lemma}\label{lem:conservative_and_compactness}
Let $X$ and $Y$ be smooth varieties and $\Phi\colon \Dqcoh{X}\to \Dqcoh{Y}$ a functor given by a kernel in $D(X \times Y)$. If the restriction of $\Phi$ to $D(X)$ is conservative then for any $\sF \in \Dqcoh{X}$, $\Phi(\sF) \in D(Y)$ if and only if $\sF \in D(X)$.
\end{Lemma}

\begin{proof}
It suffices to show the statement for the infinity-categorical enhancements of the corresponding categories (the original statement then follows by taking homotopy categories). If $X$ is smooth, then $D(X) = \mathrm{Perf}(X)$ can be identified with the subcategory of compact objects of $\Dqcoh{X}$, and $\Dqcoh{X}$ is the ind-completion of $D(X)$ as stable $\infty$-categories (see for instance \cite{BFN}).

Let us first show that $\Phi$ is conservative on all of $\Dqcoh{X}$. Equivalently we show that $\Phi^L(\Dqcoh{Y})$ weakly generates $\Dqcoh{X}$. Suppose $\sF \in \Dqcoh{X}$ and $\Hom(\Phi^L(\sG), \sF) = 0$ for all $\sG \in \Dqcoh{Y}$. We know that $\Phi^L(D(Y))$ weakly generates $D(X)$. Thus the ind-completion of $\Phi^L(D(Y))$ contains $D(X)$ and hence all of $\Dqcoh{X}$. Thus $\Phi^L(D(Y))$ weakly generates $\Dqcoh{X}$.

Suppose now $\sF \in \Dqcoh{X}$ with $\Phi(\sF) \in D(Y)$. We can write $\sF$ as a filtered (homotopy) colimit of compact objects, $\sF = \colim_\alpha \sF_\alpha$. If $\Phi(\sF) = \colim_\alpha \Phi(\sF_\alpha)$ is compact then there exists an index $\beta$ such that $\Phi(\sF_\beta) \to \Phi(\sF)$ is an isomorphism. Since $\Phi$ is conservative this implies that $\sF_\beta \to \sF$ is an isomorphism and hence that $\sF$ is compact, i.e.~is an element of $D(X)$.
\end{proof}

\begin{Remark}
The proof of Lemma~\ref{lem:conservative_and_compactness} makes essential use of the fact that if $X$ is smooth then $D(X)$ consists exactly of the compact objects on $\Dqcoh{X}$. If the underlying spaces are not smooth, then this is no longer true. On the other hand, $D(X)$ is always the compact objects in the category of ind-coherent sheaves. Thus for general $X$ and $Y$ one can simply replace quasi-coherent sheaves by ind-coherent sheaves in the above arguments.
\end{Remark}

We finish this section with some technical results that we need about t-structures.

\begin{Lemma}[{\cite[Lemma~1.1.1]{Po}}]\label{lem:t-structure:conservative}%
Let $\Phi \colon \sC_1 \to \sC_2$ be a conservative t-exact functor between triangulated categories. Then,
    \begin{align*}
        \sC_1^{\le 0} & = \bigl\{ X \in \sC_1 : \Phi(X) \in \sC_2^{\le 0} \bigr\},\\
        \sC_1^{\ge 0} & = \bigl\{ X \in \sC_1 : \Phi(X) \in \sC_2^{\ge 0} \bigr\}.
    \end{align*}
\end{Lemma}

We recall that if $(\Phi, \Psi)$ is a pair of adjoint functors, then $\Phi$ is right t-exact if and only if $\Psi$ is left t-exact. In particular an equivalence of categories is t-exact if and only if its inverse is t-exact.

\begin{Lemma}\label{lem:t-structure_and_direct_sums}
Let $(\sC^{\le 0}, \sC^{\ge 0})$ be a t-structure on $\sC$ and suppose $X,Y \in \sC$. If $X \oplus Y \in \sC^{\le 0}$ then already $X, Y \in \sC^{\le 0}$ and likewise if $X \oplus Y \in  \sC^{\ge 0}$ then $X, Y \in \sC^{\ge 0}$.
\end{Lemma}
\begin{proof}
We prove the first claim (the second is similar). Consider the distinguished triangle
$$X \to X \oplus Y \to Y \xrightarrow{0}.$$
Applying $\tau^{>0}$ we get the distinguished triangle
$$\tau^{>0}(X) \to 0 \to \tau^{>0}(Y) \xrightarrow{0}$$
which implies that $\tau^{>0}(X) = \tau^{>0}(Y) = 0$.
\end{proof}

\begin{Lemma} \label{lem:t-structure:triangle-with-shifts}%
Let $(\sC^{\le 0}, \sC^{\ge 0})$ be a t-structure on $\sC$ and suppose we have a distinguished triangle 
$$ X \to Y \to Y[k] $$
for some $k > 0$ such that $X \in \sC^{\le 0}$ and $\Hom^i(Y,Y) = 0$ for $i \gg 0$. Then $Y \in \sC^{\le 0}$.
\end{Lemma}
\begin{proof}
Let $\phi\colon Y \to Y[k]$ be the second map in the triangle and consider the composition 
$$Y \xrightarrow{\phi} Y[k] \xrightarrow{\phi[k]} Y[2k].$$ 
From the octahedral axiom we have a triangle 
$$\Cone(\phi) \to \Cone(\phi^2) \to \Cone (\phi[k]).$$
As $\Cone(\phi) = X[1]$ and $\Cone(\phi[k]) = X[k+1]$ are in $\sC^{\le -1}$, so is $\Cone(\phi^2)$. Repeating this argument we find that $\Cone(\phi^{2^n})$ is in $\sC^{\le -1}$ for all $n \ge 1$. But $\phi^{2^n} = 0$ for $n \gg 0$ which means that
$$\Cone(Y \xrightarrow{0} Y[2^nk]) \in \sC^{\le -1}.$$
But the left side is just $Y[1] \oplus Y[2^nk]$. It follows that $Y[1] \in \sC^{\le -1}$ and hence $Y \in \sC^{\le 0}$.
\end{proof}

\begin{Lemma}\label{lem:t-structures:inclusion-implies-identical}
    Suppose we have two t-structures $(\sC^{\le_1 0},\, \sC^{\ge_1 0})$ and $(\sC^{\le_2 0},\, \sC^{\ge_2 0})$ on $\sC$ such that $\sC^{\le_1 0} \subseteq \sC^{\le_2 0}$ and $\sC^{\ge_1 0} \subseteq \sC^{\ge_2 0}$. Then the two t-structures are identical.
\end{Lemma}

\begin{proof}
    In any t-structure $\sC^{\ge 0}$ is the right orthogonal of $\sC^{\le -1}$.
    Hence, $\sC^{\le_1 -1} \subseteq \sC^{\le_2 -1}$ implies $\sC^{\ge_1 0} \supseteq \sC^{\ge_2 0}$.
    Similarly one shows that $\sC^{\le_1 0} \supseteq \sC^{\ge_1 0}$.
\end{proof}

\subsection{Induced braid groups actions}\label{sec:braidaction}

One of the applications of categorical $\g$ actions is to the construction of braid group actions. We recall these results following \cite{CK2,C2}. 

Given an $\sl_2$ action generated by $\E,\F$ one considers the Rickard complexes 
\begin{align*}
[\F^{(\l)} \to \E \F^{(\l+1)} \la 1 \ra \to \E^{(2)} \F^{(\l+2)} \la 2 \ra \to \dots] & \quad \text{ if } \l \ge 0, \\
[\E^{(-\l)} \to \F \E^{(-\l+1)} \la 1 \ra \to \F^{(2)} \E^{(-\l+2)} \la 2 \ra \to \dots][\l]\la -\l \ra & \quad \text{ if } \l \le 0,
\end{align*}
where the left hand terms are in cohomological degree zero and where the differentials can be defined using the fact that $\E$ and $\F$ are adjoint to each other. These complexes live in $\Kom(\K)$ and we denote them $[\T]$. On the other hand, as is explained in \cite{CKL2}, if $\K$ is triangulated then $[\T]$ has a unique convolution which we denote $\T$. This allows us to avoid having to pass from $\K$ to $\Kom(\K)$ if $\K$ is already triangulated. 

\begin{Remark}\label{rem:convention}
For technical reasons we use the opposite convention from \cite{CKL2,CK2,CK4}. Namely, the complexes for $\T$ go to the right rather coming from the left. This means that $\T$ in this paper corresponds to $\T^{-1}$ in the old notation. In particular, this means that one needs to take the ``inverse'' of a number of results that we use from \cite{CK4}. 
\end{Remark}

More generally, given a $(L\gl_n,\t)_{KM}$ action we get such a $[\T_i]$ for each $i \in \hI$. Following \cite{CK2} these satisfy the standard braiding relations
\begin{align*}
    [\T_i][\T_j] \cong [\T_j][\T_i] & \text{ if } \la i,j \ra = 0 \\
    [\T_i][\T_j][\T_i] \cong [\T_j][\T_i][\T_j] & \text{ if } \la i,j \ra = -1.
\end{align*}
The braiding property is a consequence of the following more fundamental relations 
\begin{equation} \label{eq:TiTjEi}
[\T_i] [\T_j] \E_i \cong \E_j [\T_i] [\T_j] \qquad\text{and}\qquad [\T_i] [\T_j] \F_i \cong \F_j [\T_i] [\T_j]
\end{equation}
for $\la i,j \ra = -1$ \cite[Corollary~7.3]{C2}. The same relations hold for $\T_i,\T_j$. Note that if we denote by $s_i$ ($i \in \hI$) the generators of the affine Weyl group $\hW_n$ then $\T_i \1_\uk = \1_{s_i \cdot \uk} \T_i$.

It is also useful to consider the following ``shifted'' braid group action. We set $\T'_i \1_\uk \coloneqq \T_i \1_\uk [k_i] \la -k_i \ra$. This notation agrees with that from \cite{C2,CK4} (with the same caveat as before that we have switched $\T'$ with its inverse). We will mostly use $\T_i'$ in the rest of the paper.

These $\T'_i$ have several nice properties. First, the corresponding complexes always lie in non-positive cohomological degrees. Second, they still satisfy the relations described in \eqref{eq:TiTjEi}. Third, we have 
\begin{equation}\label{eq:TE=FT}
\F_i^{(p)} \T_i' \1_\uk \cong \T_i' \E_i^{(p)} \1_\uk \la p(\la \uk, \alpha_i \ra + p) \ra.
\end{equation}
This relation follows from \cite[Cor. 4.6]{C2} by keeping in mind that our $\T_i'$ is a shifted version of $\T_i^{-1}$ from that paper. Notice that there is no cohomological shift on the right hand side of \eqref{eq:TE=FT}. Finally, the following is an important observation. 

\begin{Lemma}\label{lem:canonical}
The finite braid group action generated by $\T_i'$ for $i \in I$ canonically identifies any two objects $\K(\uk), \K(\uk')$ if $\uk'$ can be obtained from $\uk$ by rearranging the position of some zeros. 
\end{Lemma}
\begin{proof}
The key (and fairly clear) observation is that $(\T'_i)^2 \1_\uk = \1_\uk$ if $k_i=0$ or $k_{i+1}=0$. This means that the braid group action discussed above descends to a symmetric group action. The result now follows. 
\end{proof}

\subsection{Symmetric/skew sides and some terminology}\label{sec:our_actions}

For the applications we have planned, the objects $\K(\uk)$ will be $\bZ$-graded triangulated categories, where we denote the internal grading by $\{1\}$. Our goal is to construct t-structures on these triangulated categories. The main example to keep in mind is where $\K(\uk)$ is the derived category of coherent sheaves $D(Y(\uk))$ on some variety $Y(\uk)$ and where $\Hom(\uk,\uk')$ is $D(Y(\uk) \times Y(\uk'))$ (the space of kernels). The space of 2-morphisms is then just the space of maps between two kernels. See Section \ref{sec:notation} for a brief review of kernels. 

We will always assume that $\K(\uk)$ is zero if $k_i < 0$ for some $i$. Moreover, we will fix $N$ and assume that $\K(\uk)$ is zero if $\sum_i k_i \ne N$. For convenience we will also take $n > N$. One can show this is always possible to do by extending any $L\gl_n$ action to one of $L\gl_{n+k}$ for any $k \ge 0$. We do not explain this in detail since in all our examples it will be obvious how to ensure that $n > N$. 

We will consider two cases depending on the evaluation of the grading shift $\la 1 \ra$ from the $(L\gl_n,\t)$ action. On the \enquote{symmetric} side we will have $\la 1 \ra = \{1\}$ while on the \enquote{skew} side $\la 1 \ra = [1] \{-1\}$. To more easily distinguish which side we are working on we will use $\tK(\uk)$ to denote the symmetric side and $\K(\uk)$ the skew. The reason for these names is that in our main application we have 
\begin{align*}
K(\tK(\uk)) &\cong K_{\C^\times}(\bY(\uk)) \cong \Sym^{k_1}(V) \otimes \dots \otimes \Sym^{k_n}(V) \\
K(\K(\uk)) &\cong K_{\C^\times}(Y(\uk)) \cong \Lambda^{k_1}(V) \otimes \dots \otimes \Lambda^{k_n}(V)
\end{align*}
where $K_{\C^\times}(X)$ denotes the complexified Grothendieck group of coherent sheaves on $X$ tensored over $\bZ[q,q^{-1}]$ with $\C(q)$ while $Y(\uk)$ and $\bY(\uk)$ are certain flag-like varieties and $V$ is the standard representation of $SL_m$ (for some fixed $m$). 

Suppose now we have a fixed a t-structure on $\oplus_{\uk' \in W_n \cdot \uk} \K(\uk')$ where $W_n$ is the finite Weyl group. We will say that the t-structure is \emph{braid positive} at $\uk$ if the endofunctor $\T'_i$ acting on $\oplus_{\uk' \in W_n \cdot \uk} \K(\uk')$ is right t-exact for every $i \in \hI$. Similarly, we say that a t-structure on $\oplus_\uk \K(\uk)$ is braid positive if $\T'_i$ acting on it is right t-exact for every $i \in \hI$. 

Since $n > N$ each nonzero weight space $\K(\uk)$ has at least one $k_j = 0$. By Lemma~\ref{lem:canonical} let us assume that $j=n$. We can then define 
$$\R' \1_\uk \coloneqq \T'_{n-2} \dots \T'_1\T_0' \1_\uk.$$
One can readily check that $\R' \1_\uk = \1_{r \cdot \uk} \R'$ where $r \cdot (k_1, k_2, \dots, k_n) = (k_2, \dots, k_n, k_1)$. Moreover, $\E_i \R' = \R' \E_{i+1}$ and $\F_i \R' = \R' \F_{i+1}$ for all $i \in \hI$. 

Finally, we denote by $\eta$ the \enquote{highest} weight $(0,\dots,0,N)$ and by $\mu$ the \enquote{middle} weight $(0^{n-N},1^N)$.

\section{Structure results}

In this section we collect some general results regarding an action with a target 2-category $\K$ as in Section \ref{sec:our_actions}. In particular, the objects in this category are indexed by sequences $\uk \in \bZ^n$. Recall that $\1_\uk = 0$ if $k_i < 0$ for some $i$ or if $\sum_i k_i \ne N$. The main result, Proposition \ref{prop:highest}, identifies the possible 1-endomorphisms of $\K(\eta)$ where $\eta = (0,\dots,0,N)$ is the highest weight.  

\subsection{A remark about trees}\label{sec:trees}

In \cite{CKM} one used skew Howe duality to give a description of the category of $SL_m$-representations (more precisely, the full subcategory generated by tensor products of fundamental representations) in terms of webs. For our purposes, the main point is that webs on an annulus describe precisely the $U_q(L\gl_n)_{KM}$ relations in an action on a 2-category $\K$ as above. 

\begin{align}\label{fig:1}
\begin{tikzpicture}[baseline=-1cm,y={(0,-1)}]
\foreach \x/\y in {0/0,1/0,2/0,0/1,1/1,0/2} {
        \coordinate(z\x\y) at (\x+\y/2,\y/1.5);
}
\coordinate (z03) at (1,2);
\draw[mid<] (z00) node[above] {$1$} --  (z01);
\draw[mid<] (z01) -- node[left] {$2$} (z02);
\draw[mid<] (z10) node[above] {$1$} -- (z01);
\draw[mid<] (z20) node[above] {$1$} -- (z02);
\draw[mid<] (z02) -- (z03) node[below] {$3$};
\end{tikzpicture}
 =
\begin{tikzpicture}[baseline=-1cm,y={(0,-1)}]
\foreach \x/\y in {0/0,1/0,2/0,0/1,1/1,0/2} {
        \coordinate(z\x\y) at (\x+\y/2,\y/1.5);
}
\coordinate (z03) at (1,2);
\draw[mid<] (z00) node[above] {$1$} --  (z02);
\draw[mid<] (z10) node[above] {$1$} -- (z11);
\draw[mid<] (z20) node[above] {$1$} -- (z11);
\draw[mid<] (z11) -- node[right] {$2$} (z02);
\draw[mid<](z02) -- (z03) node[below] {$3$};
\end{tikzpicture}
\end{align}
For example, the planar webs above describe equivalent 1-morphisms from $(3)$ to $(1,1,1)$ (the web is read from bottom to top). We will not recall the definition of webs here because we do not use them in a substantial way. However, the relationship is conceptually helpful. For example, any planar web between $(N)$ and $(1,\dots,1)$ can be decomposed as a direct sum of trees like the one in diagram (\ref{fig:1}). Moreover, any two such trees (1-morphisms) are equivalent because of relations like those illustrated in the diagram. 

Translating this into our language, this means that any 1-morphism between $\eta = (0,\dots,0,N)$ and $\mu = (0^{n-N},1^N)$ that does not involve any $\F_0$ or $\E_0$ (i.e. the corresponding web is planar) can be decomposed as a direct sum of 1-morphisms equivalent to $\cdots \F_{n-3}^{(N-3)} \F_{n-2}^{(N-2)} \F_{n-1}^{(N-1)}$. We call such a morphism a tree (and likewise for maps from $\mu$ and $\eta$). 

One can also consider a collection of side-by-side trees. For example, figure (\ref{fig:2}) depicts a 1-morphism from $(0,0,3,0,2)$ to $(1,1,1,1,1)$. As above, the space of planar webs from $(\uk)$ to $(1,\dots,1)$ is spanned by any choice of a collection of trees. 

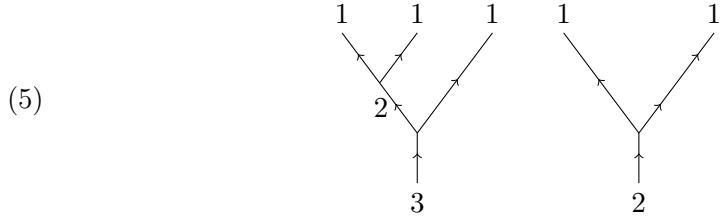
\begin{align}\label{fig:2}
\begin{tikzpicture}[baseline=-1cm,y={(0,-1)}]
\foreach \x/\y in {0/0,1/0,2/0,0/1,1/1,0/2} {
        \coordinate(z\x\y) at (\x+\y/2,\y/1.5);
}
\coordinate (z03) at (1,2);
\draw[mid<] (z00) node[above] {$1$} --  (z01);
\draw[mid<] (z01) -- node[left] {$2$} (z02);
\draw[mid<] (z10) node[above] {$1$} -- (z01);
\draw[mid<] (z20) node[above] {$1$} -- (z02);
\draw[mid<] (z02) -- (z03) node[below] {$3$};
\end{tikzpicture}
\hspace{.5cm}
\begin{tikzpicture}[baseline=-1cm,y={(0,-1)}]
\foreach \x/\y in {0/0,1/0,2/0,0/1,1/1,0/2} {
        \coordinate(z\x\y) at (\x+\y/2,\y/1.5);
}
\coordinate (z03) at (1,2);
\draw[mid<] (z00) node[above] {$1$} --  (z02);
\draw[mid<] (z20) node[above] {$1$} -- (z11);
\draw[mid<] (z11) -- node[right] {$ $} (z02);
\draw[mid<](z02) -- (z03) node[below] {$2$};
\end{tikzpicture}
\end{align}

Finally, one can consider such trees embedded on the annulus rather than the plane. We call these annular trees rather than planar trees. These describe 1-morphisms that also involve $\E_0$ and $\F_0$ (in contrast to planar trees). 

\subsection{Endomorphisms of the highest weight}

For $-N \le \ell \le N$ we set
$$\A^{(\ell)} = 
\begin{cases}
\F_0^{(\ell)} \F_1^{(\ell)} \dotsb \F_{n-1}^{(\ell)} & \text{ if } \ell \ge 0 \\
\E_{n-1}^{(-\ell)} \E_{n-2}^{(-\ell)} \dotsb \E_0^{(-\ell)} & \text{ if } \ell \le 0
\end{cases}$$
Recall that $\eta$ denotes the highest weight $(0,\dots,0,N)$. Note that $\A^{(\ell)} \1_\eta$ begins and ends at $\eta$. In terms of webs $\A^{(\pm N)}$ corresponds to twisting the strand labeled $N$ around the annulus (clockwise or counterclockwise depending on the sign). In particular, this means that $\A^{(N)}$ and $\A^{(-N)}$ are inverses of each other. The following is the main result in this section. 

\begin{Proposition}\label{prop:highest}
Any 1-endomorphism of $\K(\eta)$ is a direct summand of an element of the algebra generated by $\A^{(\ell)} \1_\eta$ for $-N \le \ell \le N$ and $\la \pm 1 \ra$.
\end{Proposition}
\begin{Remark}
It actually suffices to take $\A^{(\ell)}$ only for $1 \le \ell \le N$ and $\ell = -N$. 
\end{Remark}

In the remainder of this section we will prove this result. For any two 1-morphisms $\X$, $\Y\colon \K(\uk) \to \K(\uk')$ we write $\X \dsss \Y$ if $\X$ is a direct summand of $\bigoplus_f \Y$ for some $f \in \N[q,q^{-1}]$. For example, we have $\F_i^{(a)} \F_i^{(b)} \dsss \F_i^{(a+b)}$. More generally, if $\{\Y_\alpha : \alpha \in S\}$ is any finite collection of 1-morphisms $\K(\uk) \to \K(\uk')$ we write
$$ \X \dsss \{\Y_\alpha : \alpha \in S\} $$
if $\X$ is a direct summand of $\bigoplus_\alpha \bigoplus_{f_\alpha} \Y_\alpha$ for some $f_\alpha \in \N[q,q^{-1}]$. Note that $\dsss$ is a transitive relation.

\begin{Lemma}\label{lem:web:EFcommute}
Let $\uk$ be any weight and $a,b \in \N$. Then
\begin{align*}
\E_i^{(a)}\F_i^{(b)}\1_\uk &\dsss \bigl\{ \F_i^{(b-\ell)}\E_i^{(a-\ell)}\1_\uk : \ell \ge 0 \bigr\} \\
\F_i^{(b)}\E_i^{(a)}\1_\uk &\dsss \bigl\{ \E_i^{(a-\ell)}\F_i^{(b-\ell)}\1_\uk : \ell \ge 0 \bigr\}
\end{align*}
where $\F_i^{(c)}=\E_i^{(c)} = 0$ for $c < 0$.
\end{Lemma}
\begin{proof}
Let us prove the first claim (the second is similar). Since $\dsss$ is a transitive relation and $\F_i^{(a)} \F_i^{(b)} \dsss \F_i^{(a+b)} \dsss \F_i^{(a)} \F_i^{(b)}$ (and likewise for $\E$s) it suffices to prove the case $a=b=1$. 

If $\la \uk, \alpha_i \ra \ge 0$ then the result is clear since $\E_i \F_i \1_\uk \cong \F_i \E_i \1_\uk \bigoplus_{[\la \uk, \alpha_i \ra]} \1_\uk$. Otherwise, if $\la \uk, \alpha_i \ra \le 0$ then $\E_i \F_i \1_\uk$ is a direct summand of $\F_i \E_i \1_\uk$ and we are done again. 
\end{proof}

\begin{Lemma}\label{lem:web:sl3}
Let $\uk$ be any weight, $\la i, j\ra = -1$ and $a,b,c \in \N$. Then
$$\F_i^{(a)}\F_j^{(b)}\F_i^{(c)}\1_\uk \dsss \{\F_j^{(\ell)}\F_i^{(a+c)}\F_j^{(b-\ell)}\1_\uk : 0 \le \ell \le b\}.$$
\end{Lemma}
\begin{proof}
We use nested inductions on $b$ and then on $a$. The case $b = 1$ is \cite[Corollary~4.8]{CK2}. So assume that $b \ge 2$. 

If $a = 1$, then $\F_i\F_j^{(b)}\F_i^{(c)} \dsss \F_j^{(b-1)}\F_i\F_j\F_i^{(c)}$ by \cite[Corollary~4.8]{CK2}. By the base case $b=1$ this is then a direct summand of a direct sum of shifts of $\F_j^{(b-1)}\F_i^{(c+1)}\F_j$ and $\F_j^{(b-1)}\F_j\F_i^{(c+1)} \dsss \F_j^{(b)}\F_i^{(c+1)}$.
    
Now let $a > 1$. Then
$$\F_i^{(a)}\F_j^{(b)}\F_i^{(c)} \dsss \F_i^{(a)}\F_j\F_j^{(b-1)}\F_i^{(c)} \dsss \F_i^{(a-1)}\F_j\F_i\F_j^{(b-1)}\F_i^{(c)}.$$ 
By induction on $b$ applied to $\F_i\F_j^{(b-1)}\F_i^{(c)}$ this is a direct summand of a sum of shifts of
$$\F_i^{(a-1)}\F_j\F_j^{(\ell)}\F_i^{(c+1)}\F_j^{(b-1-\ell)} \dsss \F_i^{(a-1)}\F_j^{(1+\ell)}\F_i^{(c+1)}\F_j^{(b-1-\ell)}, \qquad 0 \le \ell \le b-1.$$
If $\ell = b-1$, this is of the required form by induction on $a$. Otherwise it is of the required form by induction on $b$. 
\end{proof}

By an \emph{$\F$-string} we will mean a sequence of the form $\F_{i_m}^{(\ell_m)} \dotsc \F_{i_1}^{(\ell_1)}$ for $0 \le i_j \le n-1$ and $\ell_j \in \N$. We call $m$ the \emph{length} of the $\F$-string.

\begin{Lemma}\label{lem:web:only-F}
Let $\X \1_\eta$ be any 1-endomorphism $\K(\eta)$. Then $\X \1_\eta \dsss \{(\A^{(-N)})^{a_\Y} \Y \1_\eta\}$ for a finite collection of $a_\Y \ge 0$ and $\F$-strings $\Y$.
\end{Lemma}

\begin{proof}
We use induction on the number of $\E$'s in $\X$. If there are none, we are done. Otherwise put all $\E$s to the left of the sequence using Lemma \ref{lem:web:EFcommute}. This does not increase the number of $\E$s. Thus we can assume $\X = \E_{n-1}^{(\ell)} \X'$ where $\X'$ has fewer $\E$s. But then 
$$\A^{(N)} \X \cong \F_0^{(N)} \dotsb \F_{n-2}^{(N)} \F_{n-1}^{(N)} \E_{n-1}^{(\ell)} \X' \cong \F_0^{(N)} \dotsb \F_{n-2}^{(N)} \F_{n-1}^{(N-\ell)} \X'$$
and the result follows by induction. 
\end{proof}

We say an $\F$-string $\X = \F_{i_m}^{(\ell_m)} \dotsc \F_{i_1}^{(\ell_1)}$ is \emph{strictly ordered} if $i_{j+1} = i_j-1 \bmod n$ and $\ell_j \ne 0$ for all $j$. In this case we write $\#_0\X$ for the number of $j$ with $i_j=0$ (the number of $\F_0^{(\ell)}$ in the string). For $\ell_0$, \dots, $\ell_{n-1} \in \N$ let $\F^{(\ell_0,\dotsc,\ell_{n-1})}$ denote $\F_0^{(\ell_0)} \dotsc \F_{n-1}^{(\ell_{n-1})}$. For simplicity in the rest of this section we will take the all indices of $\F_i$ modulo $n$.

\begin{Lemma}\label{lem:web:strictly-order-induction}
    Suppose $\X$ is a strictly ordered string. Then $\1_\eta \X \F_j^{(a)} \dsss \{ \1_\eta \Y\}$ where $\Y$ belongs to a finite collection of strictly ordered strings. Moreover, $\#_0\Y \le \#_0\X$ unless $j=0$ and $\X \cong \X'\F_{n-1}^{(b)}$ in which case $\#_0\Y = \#_0\X+1$. 
\end{Lemma}

\begin{proof}
We can assume $\1_\eta \X = \1_\eta \F_0^{(\cdot)} \F_1^{(\cdot)} \cdots \F_{n-1}^{(\cdot)} \F_0^{(\cdot)} \cdots \F_i^{(\cdot)}$ where the $(\cdot)$ represent nonzero exponents. Note that $\1_\eta \X$ has to begin on the left with a $\F_0^{(\cdot)}$ because otherwise it is zero. Note also that there might be many appearances of $\F_0^{(\cdot)}$ in the string for $\X$.

Our argument is by induction on the length of $\X$. If $j=i$ or $j=i+1$ we are done. Otherwise we start moving the $\F_j^{(a)}$ in $\1_\eta \X \F_j^{(a)}$ to the left until it hits an $\F_{j+1}$. This way we end up with 
\begin{equation}\label{eq:1}
\1_\eta \cdots \F_j^{(b)} \F_{j+1}^{(c)} \F_j^{(a)} \cdots \F_i^{(\cdot)}.
\end{equation}
Now, by Lemma \ref{lem:web:sl3}, we get
\begin{equation}\label{eq:2}
\1_\eta \X \F_j^{(a)} \dsss \{\1_\eta \cdots \F_{j+1}^{(\ell)} [\F_j^{(a+b)} \F_{j+1}^{(c-\ell)} \cdots \F_i^{(\cdot)}]: 0 \le \ell \le c \}.
\end{equation}
If $c = \ell$ then we are done by induction on the length. Otherwise the expression in the brackets is strictly ordered and we continue moving the $\F_{j+1}^{(\ell)}$ to the left. This procedure must end since the string is finite.

Finally we need to explain why the number of $\F_0^{(\cdot)}$ does not increase (unless $j=0$). In the argument above the only time an extra $\F_0^{(\cdot)}$ appears is when, in equation \eqref{eq:1}, $j=n-1$. In this case the right side of equation \eqref{eq:2} reads
$$\{\1_\eta \cdots \F_{0}^{(\ell)} [\F_{n-1}^{(a+b)} \F_{0}^{(c-\ell)} \cdots \F_i^{(\cdot)}]: 0 \le \ell \le c \}.$$
We now continue moving the $\F_0^{(\ell)}$ to the left. It will eventually hit an $\F_1^{(\cdot)}$ at which point we will have something of the form $\1_\eta \cdots \F_0^{(\cdot)} \F_1^{(\cdot)} \F_0^{(\ell)} \cdots$  which, upon simplifying again, decreases the number of $\F_0^{(\cdot)}$. So overall the number of $\F_0^{(\cdot)}$ does not increase. The only exception is if $j=0$ and $i=n-1$ in which case the algorithm immediately stops and we get an extra $\F_0^{(\cdot)}$. 
\end{proof}

\begin{Corollary}\label{cor:web:strictly-order}
For any $\F$-string $\X$ we have $\1_\eta \X \dsss \{ \1_\eta \Y \}$ where $\Y$ varies over a finite collection of strictly ordered $\F$-strings.
\end{Corollary}
\begin{proof}
This follows by applying Lemma~\ref{lem:web:strictly-order-induction} repeatedly starting on the left. 
\end{proof}

\begin{Lemma}\label{lem:webs:F-reduced}
Let $\X$ be an $\F$-string. Then
$$\1_\eta \X \1_\eta  \dsss \big\{\A^{(\ell_1)} \cdots \A^{(\ell_r)}\1_\eta \bigr\}$$
where $r \ge 0$ and $\sum_i |\ell_i| \le w_{\X}$ for some fixed $w_{\X}$ which depends on $\X$.
\end{Lemma}
\begin{proof}
By Corollary \ref{cor:web:strictly-order} we can assume that $\X$ is strictly ordered.  If $\#_0\X \le 1$ then $\1_\eta \X \1_\eta = 0$ is either zero or equal to $\1_\eta \X \1_\eta = \A^{(\ell)}\1_\eta$ for some $\ell$. So assume that $\#_0\X \ge 2$ and proceed by induction on $\#_0\X$. 

Write $\X \1_\eta = \X' \F^{(b_0,\dotsc,b_{n-1})}\F^{(a_0,\dotsc,a_{n-1})}\1_\eta$ for some strictly ordered $\F$-string $\X'$ with $a_j \ne 0$ and $b_j\ne 0$ for all $j$. Further we can assume that $a_0 \le a_1 \le \dotsb \le a_{n-1}$ as otherwise the functor vanishes. We will now induct on the sum $a_0 + \dotsb + a_{n-1}$.
    
If $a_0 = a_1 = \dotsb = a_{n-1}$, then $\F^{(a_0,\dotsc,a_{n-1})} = \A^{(a_0)}$ so that $\X'\F^{(b_0,\dotsc,b_{n-1})}\A^{(a_0)}$ and we are done by induction on $\#_0\X$. Otherwise let $j$ be the largest index such that $a_j \ne a_{j-1}$. We have to distinguish whether $j = n-1$ or not.

First assume that $0 < j < n-1$.  Let $\uk=(a_1-a_0,\dotsc,a_{n-1}-a_{n-2},N+a_0-a_{n-1})$, i.e.~such that $\1_{\uk}\F^{(a_0,\dotsc,a_{n-1})}\1_\eta$.  Then we have $k_j = a_j - a_{j-1} > 0$ and $k_{j+1} = a_{j+1} - a_j = 0$ so that $\1_\uk \dsss \F_i\E_i$:
\begin{align*}
    & \X'\F^{(b_0,\dotsc,b_{n-1})}\F^{(a_0,\dotsc,a_{n-1})}\1_\eta \\
    & \qquad \dsss \X'\F^{(b_0,\dotsc,b_{n-1})}\F_j\E_j\F^{(a_0,\dotsc,a_{n-1})}\1_\eta\\
    & \qquad = \bigr[\X'\F^{(b_0,\dotsc,b_{n-1})}\bigl]\F_j\F^{(a_0,\dotsc, a_{j-1}, a_j-1, a_{j+1},\dotsc,a_{n-1})}\1_\eta\\
    & \qquad \dsss \bigl\{ \Y\F^{(a_0,\dotsc, a_{j-1}, a_j-1, a_{j+1},\dotsc,a_{n-1})}\1_\eta \bigr\},
\end{align*}
where by Lemma~\ref{lem:web:strictly-order-induction} $\Y$ belongs to a finite set of strictly ordered $\F$-strings with $\#_0\Y \le \#_0\X' + 1 = \#_0\X-1$.
Thus either we are done by induction on $\#_0\X$, or otherwise by induction on $\sum a_j$. 
    
If $j = n-1$, then, using Lemma~\ref{lem:web:sl3}, we have
\begin{align*}
    & \X'\F^{(b_0,\dotsc,b_{n-1})}\F^{(a_0,\dotsc,a_{n-1})}\1_\eta \\
    & \qquad = \X'\F_0^{(b_0)}\dotsb\F_{n-1}^{(b_{n-1})}\F_0^{(a_0)}\dotsb\F_{n-1}^{(a_{n-1})}\1_\eta \\
    & \qquad = \X'\F_0^{(b_0)}\dotsb \bigl[\F_{n-1}^{(b_{n-1})}\F_0^{(a_0)}\F_{n-1}^{(a_{n-1}-a_{n-2})}\bigr]\F_{1}^{(a_1)}\dotsb\F_{n-2}^{(a_{n-2})}\F_{n-1}^{(a_{n-2})}\1_\eta \\
    & \qquad \dsss \bigl\{ \X'\F_0^{(b_0)}\dotsb\F_{n-2}^{(b_{n-2})}\F_0^{(\ell)}\F_{n-1}^{(b_{n-1} + a_{n-1}-a_{n-2})}\F_0^{(a_0-\ell)}\F_{1}^{(a_1)}\dotsb\F_{n-2}^{(a_{n-2})}\F_{n-1}^{(a_{n-2})}\1_\eta\bigr\}
\end{align*}
with $0 \le \ell \le a_0$.
If $\ell \ne 0$, we use Lemma~\ref{lem:web:strictly-order-induction} applied to $\bigl(\X'\F_0^{(b_0)}\dotsb\F_{n-2}^{(b_{n-2})}\bigr)\F_0^{(\ell)}$, to reduce each of the functors in the set to functors
$$\Y\F_{n-1}^{(b_{n-1} + a_{n-1}-a_{n-2})}\F^{(a_0-\ell,a_1,\dotsc,a_{n-2},a_{n-2})}\1_\eta$$
where $\Y$ is strictly ordered with $\#_0\Y \le \#_0\X' + 1 = \#_0\X-1$ (if $\ell = 0$ this is form is obtained trivially).
If  $\ell = a_0$, we can use Lemma~\ref{lem:web:strictly-order-induction} to reduce this further to strings with less $\F_0^{(\cdot)}$ than $\X$ and apply induction on $\#_0\X$.
Otherwise we can apply induction on $\sum a_j$.
\end{proof}

Lemma \ref{lem:webs:F-reduced} above proves Proposition~\ref{prop:highest} in the case of $\F$-strings. The general case follows by applying Lemma~\ref{lem:web:only-F}.

\section{t-structures on the symmetric side}\label{sec:sym}

In this section we work on the symmetric side. This means that we have a $(L\gl_n,\t)_{KM}$ action on a triangulated 2-category $\tK$ where the objects $\tK(\uk)$ are graded triangulated categories indexed by $\uk \in \bZ^n$ and where $\la 1 \ra = \{1\}$ is the internal grading. Since $[k]\la -k \ra = [k]\{-k\}$ the complexes $\T_i'$ now become $\T'_i \1_\uk = \T_i \1_\uk [k_i] \{-k_i\}$. Recall that $\mu$ denotes the weight $(0^{n-N},1^N)$ and $\eta$ the weight $(0,\dots,0,N)$ and that $\1_\uk = 0$ if $k_i < 0$ for some $i$ or $\sum_i k_i \ne N$. 

\begin{Theorem}\label{thm:sym:main}
Suppose that $\tK(\mu)$ is weakly generated by the image of $\tK(\eta)$ under the $L\gl_n$ action. Moreover, suppose that $\tK(\eta)$ is equipped with a t-structure such that for all $-N \le \ell \le N$ the 1-endomorphisms $\A^{(\ell)}$ of $\tK(\eta)$ are t-exact. Then one can uniquely extend this t-structure to all other categories $\tK(\uk)$ so that $\E_i$ and $\F_i$ are t-exact for all $i \in \hI$. Moreover, this t-structure is braid positive. 
\end{Theorem}
\begin{proof}
First note that every $\tK(\uk)$ is weakly generated by the image of $\tK(\eta)$. This is because for each $\tK(\uk)$ we have a 1-morphism $\psi_\uk \colon \tK(\uk) \rightarrow \tK(\mu)$ corresponding to a ``planar tree'' (see Section \ref{sec:trees}). This 1-morphism has the property that $\psi_\uk^L \circ \psi_\uk$ contains at least one copy of the identity 1-morphism of $\tK(\uk)$. So if $X \in \tK(\mu)$ weakly generates then $\psi_\uk^L(X) \in \tK(\uk)$ also weakly generates since 
$$\Hom(\psi_\uk^L(X)[i],Y) = 0 \Leftrightarrow \Hom(X[i], \psi_\uk(Y)) = 0 \Leftrightarrow \psi_\uk(Y)=0 \Leftrightarrow Y=0.$$

Next we define the t-structure on $\tK(\uk)$. Choose a 1-morphism 
$$\psi\colon \tK(\eta) \to \tK(\uk)$$
whose image weakly generates $\tK(\uk)$ and such that $\psi^R \circ \psi$ contains at least one copy of the identity 1-morphism (this is possible since we can just add into $\psi$ a copy of the planar tree morphism from $\tK(\eta)$ to $\tK(\uk)$). Now $\psi^R \circ \psi$ is an endomorphism of $\tK(\eta)$ and hence, by Proposition~\ref{prop:highest},
$$\psi^R \circ \psi \dsss \{\A^{(\ell_1)} \cdots \A^{(\ell_k)}\}$$
for some finite collection of $\ell_1, \dots, \ell_k \in \bZ$. By assumption $\A^{(\ell_1)} \cdots \A^{(\ell_k)}$ is t-exact. Now, $\la 1 \ra = \{1\}$ is t-exact and so by Lemma~\ref{lem:t-structure_and_direct_sums} it follows that $\psi^R \circ \psi$ is t-exact. 

We can now apply Theorem~\ref{thm:t-structure:right-exact-induces} to 
$$\psi^R\colon \tK(\uk) \rightarrow \tK(\eta)$$ 
to obtain an induced t-structure on $\tK(\uk)$. Since the image of $\psi$ weakly generates $\tK(\uk)$, $\psi^R$ is conservative. Thus the t-structure on $\tK(\uk)$ is given by
    \begin{equation}\label{eq:sym:definition_of_tstructure}
    \begin{split}
    \tK(\uk)^{\le 0} & = \bigl\{ A \in \tK(\uk) : \psi^R(A) \in \tK(\eta)^{\le 0} \bigr\},\\
    \tK(\uk)^{\ge 0} & = \bigl\{ A \in \tK(\uk) : \psi^R(A) \in \tK(\eta)^{\ge 0} \bigr\}.
    \end{split}
    \end{equation}

The t-structure above does not depend on the choice of $\psi$ as long as it weakly generates. This is because given two choices $\psi$, $\psi'$ we can consider the t-structures defined by the 1-morphisms $\psi$ and $\psi \oplus \psi'$ respectively. Then, using (\ref{eq:sym:definition_of_tstructure}), Lemma~\ref{lem:t-structure_and_direct_sums} implies that $\tK(\uk)^{\le 0}$ and $\tK(\uk)^{\ge 0}$ for the latter are contained in the ones for the former. Hence by Lemma~\ref{lem:t-structures:inclusion-implies-identical} the two t-structures are equal. Likewise, $\psi \oplus \psi'$ and $\psi'$ define the same t-structure. Hence $\psi$ and $\psi'$ also induce the same t-structure.

From this it follows that $\E_i$ and $\F_i$ are exact. More precisely, to show $\E_i \1_\uk$ is t-exact it suffices to show that $\psi^R \E_i \1_\uk$ is t-exact. But this is t-exact by the discussion above. Moreover, $\F_i$ is exact since it is biadjoint to $\E_i$ (up to some $\{\cdot\}$ shift which is exact). 

To see uniqueness of the t-structure suppose that there exist another t-structure on all $\tK(\uk)$ extending the t-structure on $\tK(\eta)$ such that $\E_i$ and $\F_i$ are t-exact for all $i \in \hI$. Then the functors $\psi^R$ defined above will also be t-exact. Thus \eqref{eq:sym:definition_of_tstructure} shows that the two t-structures agree.

Finally, we have to show that for this t-structure the functors $\T_i'$ are right t-exact. We assume that $k_i \le k_{i+1}$ (the case $k_i \ge k_{i+1}$ is similar). Set $\A_j = \E_i^{(k_i-j)} \F_i^{(k_{i+1}-j)}$ for $j \ge 0$. Then $\T'_i \1_\uk$ is the convolution of the complex  
$$\A_{k_i} \to \dotsb \to \A_1 \to \A_0$$
where $\A_0$ is in cohomological degree zero. Since each $\A_i$ is t-exact it follows that $\T'_i$ is right t-exact. 
\end{proof}

\section{t-structures on the skew side}\label{sec:skew}

In this section we work on the skew side. This means that we have an $(L\gl_n,\t)_{KM}$ action on a triangulated 2-category $\K$ where the objects $\K(\uk)$ are graded triangulated categories indexed by $\uk \in \bZ^n$ and where $\la 1 \ra = [1]\{-1\}$ where $[1]$ is the cohomological grading and $\{1\}$ is the internal grading. Since $[k]\la -k \ra = \{k\}$ we now have $\T_i' \1_\uk = \T_i \1_\uk \{k_i\}$. Recall that $\mu$ denotes the weight $(0^{n-N},1^N)$ and that $\1_\uk = 0$ if $k_i < 0$ for some $i$ or $\sum_i k_i \ne N$. 

\begin{Theorem}\label{thm:skew:main}
Suppose we have a braid positive t-structure on 
$$\bigoplus_{\mu' \in W_n \cdot \mu} \K(\mu').$$ 
Then there exists a unique t-structure on $\bigoplus_\uk \K(\uk)$ determined by exactness of the 1-morphism $\psi\colon \K(\uk) \rightarrow \K(\mu)$ corresponding to a collection of planar trees. Moreover, this t-structure is braid positive.
\end{Theorem}
\begin{Remark}
The 1-morphism $\psi\colon \K(\uk) \rightarrow \K(\mu)$ above is conservative because the composition $\psi^R \circ \psi$ contains at least one summand $\1_\uk$. Hence Lemma~\ref{lem:t-structure:conservative} implies that the t-structure in Theorem~\ref{thm:skew:main} is given by
\begin{align*}
        \K(\uk)^{\ge 0} &= \bigl\{ A \in \K(\uk) : \psi(A) \in \K(\mu)^{\ge 0} \bigr\} \\
        \K(\uk)^{\le 0} &= \bigl\{ A \in \K(\uk) : \psi(A) \in \K(\mu)^{\le 0} \bigr\}.
\end{align*}
\end{Remark}

In the rest of this section we will prove this theorem. 

\subsection{Some preliminaries}

Let us denote by $|\uk|$ the number of $k_i \ne 0$. The idea for proving Theorem \ref{thm:skew:main} is to use the categorical action to define a t-structure on all $\K(\uk)$ by using the given t-structure on $\K(\mu)$. We will do this by (decreasing) induction on $|\uk|$ 

\begin{Lemma} \label{lem:skew:T-squared-on-k1}
Consider $\uk$ with $k_i > 0$ and $k_{i+1} = 1$. Then we have distinguished triangles
\begin{align*}
& \T'_i \T'_i \1_\uk \to \1_\uk \{2\} \to \X \\
& \X \to \E_i \F_i \1_{\uk} [k_i] \{-k_i+2\} \to \E_i \F_i \1_{\uk}[k_i+2]\{-k_i\}.
\end{align*}
for some $\X \in \End(\K(\uk))$. 
\end{Lemma}
\begin{proof}
This follows from Proposition 6.8 of \cite{C1} where it is shown that $[\T'_i] [\T'_i] \1_\uk$ is homotopic to a complex of the form 
$$\1_\uk \{2\} \rightarrow \E_i \F_i \1_\uk [k_i]\{-k_i+2\} \rightarrow \E_i \F_i \1_\uk [k_i+2]\{-k_i\}.$$
Note that the extra shift of $\{2\}$ here is just a result of our conventions for $\T'_i$. 
\end{proof}

\begin{Corollary}\label{cor:EE_L}
Suppose again that we have $\uk$ with $k_i > 0$ and $k_{i+1} = 1$. If $(\T_i')^2 \1_\uk$ is right t-exact then $\E_i \E_i^L \1_\uk$ is right exact.
\end{Corollary}
\begin{proof}
Since $\F_i \1_{\uk} \cong \E_i^L \1_\uk [k_i]\{-k_i\}$ we get distinguished triangles
\begin{align*}
& \T'_i \T'_i \1_\uk \to \1_\uk \{2\} \to \X \\
& \X \to \E_i \E_i^L \1_{\uk} \{2\} \to \E_i \E_i^L \1_{\uk} [2]. 
\end{align*}
From the first triangle it follows that $\X$ is right exact. From the second triangle it then follows by Lemma~\ref{lem:t-structure:triangle-with-shifts} that $\E_i \E_i^L \1_{\uk}$ is also right exact.
\end{proof} 

\subsection{Defining the t-structure on $\K(\uk)$}\label{sec:K}

We will now define the t-structure on every $\K(\uk)$. This will be done inductively and using only the $\gl_n$ part of the $L\gl_n$ action (the entire action will be dealt with subsequently). 

By assumption we have a t-structure on $\K(\uk)$ where $|\uk| = N$ (this is the base case). Suppose we have inductively constructed t-structures on all $\K(\uk)$ with $|\uk| > r$ and that $\T_i'$ for $i=1,\dots,n-1$ are right t-exact for these t-structures. Now consider some $\uk$ with $|\uk| = r$. By Lemma \ref{lem:canonical} we can assume that there exists an index $i$ with $k_i > 1$ and $k_{i+1}=0$. 

Consider the 1-morphism $\E_i \1_\uk$. By induction we have a t-structure on $\K(\uk+\alpha_i)$ so that $\T'_i$ is right t-exact for $i=1,\dots,n-1$. Corollary \ref{cor:EE_L} implies that $\E_i \E_i^L \1_{\uk+\alpha_i}$ is right t-exact. Subsequently, we get an induced t-structure on $\K(\uk)$ using Theorem \ref{thm:t-structure:right-exact-induces}. Moreover, since $\E_i \1_\uk$ is conservative in this case (since $\F_i \E_i \1_\uk$ contains several copies of the identity) we find that the t-structure is given by 
\begin{align*}
\K(\uk)^{\ge 0} &= \bigl\{ A \in \K(\uk) : \E_i(A) \in \K(\uk+\alpha_i)^{\ge 0} \bigr\} \\
\K(\uk)^{\le 0} &= \bigl\{ A \in \K(\uk) : \E_i(A) \in \K(\uk+\alpha_i)^{\le 0}\bigr\}.
\end{align*}

\begin{Lemma}
The t-structure on $\K(\uk)$ defined above is independent of any choices made along the way.
\end{Lemma}
\begin{proof}
By construction, the t-structure on $\K(\uk)$ is the unique t-structure with 
$$\K(\uk)^{\ge 0} = \big\{ A \in \K(\uk): \psi(A) \in \K(\mu)^{\ge 0} \bigr \}$$
where $\psi\colon \K(\uk) \rightarrow \K(\mu)$ is the functor associated to some planar tree from $\uk$ to $\mu$. As remarked in Section \ref{sec:trees} even though this tree is not unique all such trees give the same 1-morphisms. In this sense, the resulting t-structure on $\K(\uk)$ does not depend on any choices made in defining the other $\K(\uk')$ with $|\uk'| > r$. 
\end{proof}

\begin{Corollary}\label{cor:1}
If $k_i=0$ or $k_{i+1}=0$ then $\T_i' \1_\uk$ is t-exact. 
\end{Corollary}
\begin{proof}
This is just because for 1-morphism $\psi\colon \K(s_i \cdot \uk) \rightarrow \K(\mu)$ corresponding to a planar tree, the composition $\psi \T_i' \1_\uk$ corresponds to another planar tree (since $k_i=0$ or $k_{i+1}=0$). 
\end{proof}

So now we can assume that we have t-structures for all $\K(\uk)$ with $|\uk| \ge r$. The next step is to show that $\T'_i$ for $i=1,\dots,n-1$ are right t-exact for these t-structures. 

\begin{Proposition}\label{prop:braidpos1}
Using the t-structure on $\K(\uk)$ defined above $\T'_i$ is right t-exact for $i=1,\dots,n-1$.
\end{Proposition}
\begin{proof}
Consider some $\T'_i \1_\uk$. Using Corollary \ref{cor:1} we can assume that (say) $k_{i+2}=0$. To show $\T'_i \1_\uk$ is right exact it suffices to show that $\E_{i+1} \T'_i \1_\uk \T'_{i+1}$ is, since $\1_\uk \T'_{i+1}$ is t-exact by Corollary \ref{cor:1}. But 
$$\E_{i+1} \T'_i \1_\uk \T'_{i+1} \cong \T'_i \T'_{i+1} \1_{s_{i+1} \cdot \uk + \alpha_i} \E_i \1_{s_{i+1} \cdot \uk}$$
and $\T'_i \T'_{i+1} \1_{s_{i+1} \cdot \uk + \alpha_i}$ is right exact by induction. Moreover, $\E_i \1_{s_{i+1} \cdot \uk}$ is t-exact by construction. The result follows. 
\end{proof}

\subsection{Extending to the affine case}

Recall that $\R' \1_\uk$ is defined as $\T'_{n-2} \dots \T'_1 \T'_0 \1_\uk$ if $k_n=0$ (if $k_n \ne 0$ then we use the finite braid group action to move a zero into that spot). 

\begin{Lemma}\label{lem:Rexact}
The functor $\R' \1_\uk\colon \K(\uk) \rightarrow \K(r \cdot \uk)$ is right t-exact.
\end{Lemma}
\begin{proof}
We prove this by decreasing induction on $|\uk|$. The base case follows since 
$$\R' \1_{(k_1, \dots, k_n,0)} = \T'_{n-2} \dots \T'_1 \T'_0 \1_{(k_1, \dots, k_n,0)}$$ 
and each term on the right hand side is right exact by our basic hypothesis. 

To prove the induction step consider now the composition 
$$\K(0,k_1,\dots,k_n,0) \xrightarrow{\R' \R'} \K(k_2, \dots, k_n,0,0,k_1) \xrightarrow{\F_{n+1}} \K(k_2, \dots, k_n,0,1,k_1-1).$$
The first $\R' \1_{(0,k_1,\dots,k_n,0)}$ is t-exact since $\T'_0 \1_{(0,k_1,\dots,k_n,0)}$ is the identity (and hence exact). Thus the composition is right exact if and only if $\R' \1_{(k_1,\dots,k_n,0,0)}$ is right exact. On the other hand, the composition also equals $\R' \R' \1_{(1,k_1-1,\dots,k_n,0)} \F_1$ which is right exact by decreasing induction on $|\uk|$. 
\end{proof}

\begin{Corollary}\label{cor:Rexact}
The functor $\R' \1_\uk\colon \K(\uk) \rightarrow \K(r \cdot \uk)$ is t-exact. 
\end{Corollary}
\begin{proof}
In Lemma \ref{lem:Rexact} we proved that $\R' \1_\uk$ is right exact. Note that in the description of $\R' \1_\uk$ as $\T'_{n-2} \dots \T'_1 \T'_0 \1_\uk$ each $\T'_i$ equals its inverse because it always exchanges a zero with something. Thus a similar argument to that used in Lemma \ref{lem:Rexact} also shows that $\R' \1_\uk$ is left exact. 
\end{proof}

\begin{Remark}
Notice that the proofs used in Lemma \ref{lem:Rexact} and Corollary \ref{cor:Rexact} also work to show that $\R' \1_{\uk}\colon \tK(\uk) \rightarrow \tK(r \cdot \uk)$ is t-exact on the symmetric side. 
\end{Remark}

Since any affine braid group element is a composition of $\R'$ and finite braid group elements, Proposition \ref{prop:braidpos1} and Corollary \ref{cor:Rexact} imply that our t-structures are braid positive. This completes the proof of Theorem \ref{thm:skew:main}.

\section{Combining the skew and symmetric sides}

So far we have discussed the symmetric and skew sides separately. In this section we assume that we have actions on $\oplus_\uk \tK(\uk)$ and $\oplus_\uk \K(\uk)$ as in sections \ref{sec:sym} and \ref{sec:skew}. Proposition \ref{prop:skewsym} below explains how to induce a braid positive t-structure on $\K(\uk)$ from one on $\tK(\uk)$. In this way our t-structures on $\oplus_\uk \tK(\uk)$ as well as $\oplus_\uk \K(\uk)$ are induced from a single, fixed t-structure on $\tK(\eta)$. 

\begin{Proposition} \label{prop:skewsym}
Suppose we have a conservative functor 
$$\iota: \bigoplus_{\mu' \in W_n \cdot \mu} \K(\mu') \rightarrow \bigoplus_{\mu' \in W_n \cdot \mu} \tK(\mu')$$ 
such that $\iota \circ \T'_i \cong \T'_i \circ \iota$ for all $i \in \hI$. Also suppose $\iota$ has a left adjoint $\iota^L$ such that $\iota \circ \iota^L$ is right t-exact. Then there exists a unique t-structure on $\oplus_{\mu' \in W_n \cdot \mu} \K(\mu')$ such that $\iota$ is t-exact. Moreover, this t-structure is braid positive.
\end{Proposition}
\begin{proof}
Existence of a t-structure so that $\iota$ is t-exact follows from Theorem~\ref{thm:t-structure:right-exact-induces}. On the other hand, since $\iota$ is conservative, $\T'_i$ acting on $\oplus_{\mu' \in W_n \cdot \mu} \K(\mu')$ is right t-exact if and only if $\iota \circ \T'_i \cong \T'_i \circ \iota$ is right t-exact. The result follows since $\T'_i$ acting on $\tK$ is right t-exact. 
\end{proof}

Combining Proposition \ref{prop:skewsym} with the construction from Section \ref{sec:K} we get the following Corollary. 

\begin{Corollary}\label{cor:skewsym}
Under the hypothesis of Proposition \ref{prop:skewsym}, there exists a unique t-structure on $\oplus_\uk \K(\uk)$ such that
\begin{enumerate}
\item $\iota: \oplus_{\mu' \in W_n \cdot \mu} \K(\mu') \rightarrow \oplus_{\mu' \in W_n \cdot \mu} \tK(\mu')$ is t-exact and
\item the 1-morphism $\K(\uk) \rightarrow \K(\mu)$ induced by a collection of planar trees is t-exact.
\end{enumerate}
Moreover, this t-structure on $\oplus_\uk \K(\uk)$ is braid positive.
\end{Corollary} 

\begin{Remark}
Note that in Proposition \ref{prop:skewsym} we only require a functor $\iota$ for the middle weight $\mu$. In some cases, like our examples in the next section, there are also natural functors $\K(\uk) \rightarrow \tK(\uk)$ for arbitrary $\uk$. But it seems that these functors are (in general) not as well behaved. 
\end{Remark}

\section{Examples of categorical actions}\label{sec:examples}

\subsection{Notation}\label{sec:notation}

All the varieties that we consider in this section are naturally equipped with a $GL_m[[z]] \rtimes \C^\times$ action. Let us fix a subgroup $H \subseteq GL_m[[z]]$. 

For a variety equipped with a $H \rtimes \C^\times$ action we will denote by $D(X)$ the bounded derived category of $H \rtimes \C^\times$ equivariant coherent sheaves on $X$. We denote by $\O_X \{k\}$ the structure sheaf of $X$ with non-trivial $\C^\times$ action of weight $k$.

For two smooth varieties $X,Y$ as above and an object $\sP \in D(X \times Y)$ whose support is proper over $Y$ (we will always assume this is the case from hereon) we obtain an induced functor $\Phi_{\sP}: D(X) \rightarrow D(Y)$ via $(\cdot) \mapsto \pi_{2*}(\pi_1^* (\cdot) \otimes \sP)$ where $\pi_1,\pi_2$ are the obvious projections. We say $\sP$ is the kernel which induces $\Phi_{\sP}$.

The right and left adjoints of $\Phi_{\sP}$ are induced by $\sP^R \coloneqq \sP^\vee \otimes \pi_2^* \omega_X [\dim(X)]$ and $\sP^L \coloneqq \sP^\vee \otimes \pi_1^* \omega_Y [\dim(Y)]$ respectively. Moreover, if $\sQ \in D(Y \times Z)$ then $\Phi_{\sQ} \circ \Phi_{\sP} \cong \Phi_{\sQ * \sP}\colon D(X) \rightarrow D(Y)$ where $\sQ * \sP = \pi_{13*}(\pi_{12}^* \sP \otimes \pi_{23}^* \sQ)$ is the convolution product.

\subsection{The 2-category $\tK_{\Gr,m}^{n,N}$}\label{sec:cattk}

We first recall the 2-category $\tK_{\Gr,m}^{n,N}$ which is most closely related to the geometry of the Beilinson--Drinfeld Grass\-mann\-ian. See \cite{CK4} for details. 

Fix $m \in \N^{\ge 1}$. For a sequence $\uk = (k_1, \dots, k_n)$ define the varieties 
\begin{equation*}
\bY(\uk) \coloneqq \{ \C[z]^m = L_0 \subseteq L_1 \subseteq \dots \subseteq L_n \subseteq \C(z)^m : z L_i \subset L_i, \dim(L_i/L_{i-1}) = k_i \}
\end{equation*}
where the $L_i$ are complex vector subspaces. These spaces are smooth \cite[Section 3]{CK4} and $\dim \bY(\uk) = m \sum_i k_i$. They are equipped with a natural action of $GL_m[[z]] \rtimes \C^\times$ where the $\C^\times$ action is induced by $t \cdot z = t^2 z$. As discussed above, we fix a subgroup $H \subseteq GL_m[[z]]$. Everything we consider from now on will be $H \rtimes \C^\times$ equivariant.

There exist natural vector bundles $\sV_i$ on $\bY(\uk)$ whose fiber over a point $L_\bullet$ is the vector space $L_i/L_0$. There are also correspondences given by 
\begin{equation}\label{eq:corr}
\bY_i^r(\uk) = \{ (L_\bullet, L'_\bullet) : L_j = L_j' \text{ for } j \ne i, L_i' \subset L_i \} \subset \bY(\uk) \times \bY(\uk+r\alpha_i).
\end{equation}
We take $\tK^{n,N}_{\Gr,m}$ to be the 2-category whose objects are $D(\bY(\uk))$ with $\sum_i k_i = N$, the 1-morphisms are kernels inside $D(\bY(\uk) \times \bY(\uk'))$ and the 2-morphisms are maps between kernels. 

The main result in \cite{CK4} defines an $(\Lsl_n, \t)$ action on $\tK_{\Gr,m}^{n,N}$ by using the correspondences in \eqref{eq:corr}. In fact, in \cite[Section 5.3]{CK4} one discusses a shifted action using generators $\E'_{i,1}$ and $\F'_{i,-1}$ rather than $\E_{i,1}$ and $\F_{i,-1}$ (the two differ by a grading shift). It is this shifted action that we will use from here on (see also Remark \ref{rem:primes}).

Following definitions \eqref{eq:T0def} and \eqref{eq:E0def} from the appendix, we define 
\begin{equation}\label{eq:T'0def}
\T_0' \coloneqq \phi_1' \dots \phi_{n-1}' (\T'_1 \T'_2 \dots \T'_{n-1} \dots \T'_2 \T'_1)^{-1}
\end{equation}
\begin{equation}\label{eq:E'0def}
\E_0 \coloneqq (\T'_0 \T'_1)^{-1} \E_1 \T'_0 \T'_1 \ \ \text{ and } \ \ \F_0 \coloneqq (\T'_0 \T'_1)^{-1} \F_1 \T'_0 \T'_1. 
\end{equation}
Here $\phi_i' = \T'_i \T_{i,1}'$ and $\T'_{i,1}$ is the analogue of $\T_i'$ where instead of $\E_i,\F_i$ we use $\E'_{i,1},\F'_{i,-1}$. 

\begin{Theorem}\label{thm:action1}
The functors $\E_i,\F_i$ for $i \in I$ together with $\E_0$ and $\F_0$ defined above induce an $(L\gl_n,\t)_{KM}$ action on $\tK_{\Gr,m}^{n,N}$ where the map $\spn\{\alpha_i: i \in \hI \} \rightarrow \End^2(\1_\uk)$ is defined as the composition $\spn\{\alpha_i: i \in \hI \} \xrightarrow{p} \spn\{\alpha_i: i \in I \} \rightarrow \End^2(\1_\uk)$ where 
\begin{equation}\label{p:map}
p(\alpha_i) \coloneqq \begin{cases} \alpha_i & \text{ if } i \in I \\ - \sum_{j \in I} \alpha_j & \text{ if } i = 0 \end{cases} 
\end{equation}
\end{Theorem}
\begin{proof}
It is not difficult to check that the relations from Corollary \ref{cor:app1} hold despite using $\E'_{i,1}, \F'_{i,-1}$ instead of $\E_{i,1}, \F_{i,-1}$ and $\T'_i$ instead of $\T_i$. This is just a question of the grading/cohomological shifts matching up. For example, it is an elementary exercise to check that, if $\la i,j \ra = -1$, then 
$$\T_i^{-1} \E_{j,1} \T_i \cong \T_j^{-1} \E_{i,1} \T_j \Rightarrow (\T'_i)^{-1} \E'_{j,1} \T'_i \cong (\T'_j)^{-1} \E'_{i,1} \T'_j.$$ 

This takes care of all the conditions for having an $(L\gl_n,\t)_{KM}$ action with the exception of \eqref{co:theta}. In particular, it remains to show that $\E_0^2 \cong \E_0^{(2)} \{-1\} \oplus \E_0^{(2)} \{1\}$ and that $\theta \in \spn\{\alpha_i: i \in \hI \}$ induces a zero map between the summands $\E_0^{(2)} \{1\}$ if and only if $\la \theta, \alpha_0 \ra = \la \theta, - \sum_{i \in I} \alpha_i \ra = 0$.

To clarify notation let us denote by $\gamma$ the maps $\spn\{\alpha_i: i \in \hI \} \rightarrow \End^2(\1_\uk)$. Recall that $\E_0 = (\T'_0 \T'_1)^{-1}  \E_1 \T'_0 \T'_1$. Now consider $I \gamma(\theta) I \in \End^2(\E_0 \1_\uk \E_0)$. Since
$$\E_0 \1_\uk \E_0 \cong (\T'_0 \T'_1)^{-1} \E_1 \T'_0 \T'_1 \1_\uk (\T'_0 \T'_1)^{-1} \E_1 \T'_0 \T'_1$$
this map induces a zero map between the summands $\E_0^{(2)} \{1\}$ if and only if 
$$I [(T'_0 T'_1)(\gamma(\theta))] I \in \End^2(\E_1 \1_{s_0 s_1 \cdot \uk} \E_1)$$
induces a zero map between summands $\E_1^{(2)} \{1\}$. By Lemma \ref{lem:Taction1} this is the case if and only if $\la \theta, (T_1 T_0)(\alpha_1) \ra = 0$. But 
$$T_1 T_0 (\alpha_1) = T_1 \phi_1 \dots \phi_{n-1} (T_1 T_2 \dots T_{n-1} \dots T_2 T_1) (\alpha_1) = - \sum_{i \in I} \alpha_i.$$
The result follows. 
\end{proof}

\begin{Remark}
For $i \in I$ the actions of $\E_i \1_\uk$ and $\1_\uk \F_i $ are given by kernels
\begin{align*}
\O_{\bY^1_i(\uk)} \{k_i-1\} \in D(\bY(\uk) \times \bY(\uk+\alpha_i)) \\
\O_{\bY^1_i(\uk)} \{k_{i+1}\} \in D(\bY(\uk+\alpha_i) \times \bY(\uk))
\end{align*}
The kernels corresponding to $\E_0 \1_\uk$ and $\1_\uk \F_0$ are more difficult to describe explicitly. 
\end{Remark}

\begin{Lemma}\label{lem:Taction1}
Working with $\tK^{n,N}_{\Gr,m}$ we have the commutative diagrams
\begin{align*}
\xymatrix{
\spn\{\alpha_i: i \in \hI \} \ar[d]^{\gamma} \ar[r]^{T_i} & \spn\{\alpha_i: i \in \hI \} \ar[d]^{\gamma} & \spn\{\alpha_i: i \in \hI \} \ar[d]^\gamma \ar[r]^{id} & \spn\{\alpha_i: i \in \hI \} \ar[d]^\gamma \\
\End^2(\1_\uk) \ar[r]^{T'_i} & \End^2(\1_{s_i \cdot \uk}) & \End^2(\1_\uk) \ar[r]^{\phi'_i} & \End^2(\1_{\uk}) }
\end{align*}
where the map $T_i \in \End(\spn\{\alpha_i: i \in \hI \})$ is the standard action given by reflecting in $\alpha_i$ and $T'_i,\phi_i'$ are the isomorphisms induced by the equivalences $\T'_i, \phi_i'$. 
\end{Lemma}
\begin{proof}
Recall from \cite{CK4} that we have maps 
$$\ch\colon \bY(\uk) \rightarrow \bA_\uk$$
where $\bA_{\uk} = \bA_{k_1} \times \dots \times \bA_{k_n}$ and $\bA_{k_i} = \bA^{k_i}/S_{k_i}$. The map is given by keeping track of the characteristic polynomials of $z$ acting on $L_i/L_{i-1}$ for $i=1,\dots,n$. Since the morphism $\ch$ is proper we have 
$$\Hom(\O_{\bY(\uk)}, \O_{\bY(\uk)} \{2\}) \cong \Hom(\O_{\bA_\uk}, \O_{\bA_\uk} \{2\}) \cong \spn \{\tr_i: i=1,\dots,n \}$$
where $\tr_i(L_\bullet) \coloneqq \tr(z|_{L_i/L_{i-1}})$. 

Now denote by $\bA_\uk^o \subset \bA_\uk$ the open subset where the eigenvalues of $z|_{L_n/L_0}$ are all distinct. We denote $\bY(\uk)^o \coloneqq \ch^{-1}(\bA_\uk^o)$ and define $\bZ_i^o(\uk) \subset \bY(\uk)^o \times \bY(s_i \cdot \uk)^o$ as the locus 
\begin{align*}
\{ (L_\bullet, L'_\bullet): & L_j = L'_j \text{ if } j \ne i, z|_{L_{i+1}/L_{i-1}} \text{has distinct eigenvalues,} \\
& z|_{L_{i+1}/L_i} \text{ and } z|_{L'_i/L_{i-1}} \text{have the same (distinct) eigenvalues}\}.
\end{align*}
Notice this is the graph of an isomorphism. Following \cite[Proposition 7.1]{CK4}, it is not hard to see that $\T'_i$ restricted to $\bY(\uk)^o$ is induced by $\O_{\bZ_i^o}$ (in fact, the closure of this locus induces $\T'_i$ but this is harder to prove and not necessary for our purposes). 

Denote by $\pi_1$ and $\pi_2$ the natural projections from $\bZ_i(\uk)^o$ to $\bY(\uk)^o$ and $\bY(s_i \cdot \uk)^o$. Then on $\bZ_i(\uk)^o$ it is clear that $\pi_1^*(\tr_j) = \pi_2^*(\tr_j)$ if $j \ne i,i+1$ while $\pi_1^*(\tr_i) = \pi_2^*(\tr_{i+1})$ and $\pi_1^*(\tr_{i+1})=\pi_2^*(\tr_i)$. It follows that 
$$T'_i(\tr_j) = \begin{cases} 
\tr_j &\text{ if } j \ne i,i+1 \\
\tr_{i+1} &\text{ if } j=i \\
\tr_i &\text{ if } j=i+1. \end{cases}$$ 
The commutativity of the left square now follows since $\gamma: \spn\{\alpha_i: i \in \hI \} \rightarrow \End^2(\1_\uk)$ is defined by taking $\alpha_i \mapsto - \tr_i + \tr_{i+1}$. 

Finally, the fact that $\phi'_i$ acts trivially on $\End^2(\1_\uk)$ (commutativity of the right diagram) is a consequence of \cite[Corollary 7.3]{CK4} which shows that $\phi'_i$ is given by tensoring with a particular line bundle. 
\end{proof}

\subsection{The 2-category $\K^{n,N}_{\Gr,m}$}\label{sec:catk}

The 2-category $\K^{n,N}_{\Gr,m}$ is similar to $\tK^{n,N}_{\Gr,m}$ but instead of $\bY(\uk)$ we use
$$Y(\uk) \coloneqq \{ \C[z]^m = L_0 \subseteq L_1 \subseteq \dots \subseteq L_n \subseteq \C(z)^m : z L_i \subset L_{i-1}, \dim(L_i/L_{i-1}) = k_i \}.$$
These varieties are smooth with $\dim Y(\uk) = \sum_i k_i(m-k_i)$. We still have a $GL_m[[z]] \rtimes \C^\times$ action on all these categories and, as before, we work $H$-equivariantly with respect to some fixed $H \subseteq GL_m[[z]] \rtimes \C^\times$. The constructions in \cite{CKL1,C2} give rise to an $(\Lsl_n, \t)$ action on $\K^{n,N}_{\Gr,m}$ which we now describe in detail. 

Consider the correspondences 
$$W_i^r(\uk) = \{(L_\bullet,L'_\bullet): L_j=L_j' \text{ for } j \ne i, L'_i \subset L_i\} \subset Y(\uk) \times Y(\uk+r\alpha_i).$$
Note that in this definition we require that $W_i^r(\uk)$ is a subvariety which, in particular, implies that $zL_{i+1} = zL'_{i+1} \subset L'_i$ and $zL_i \subset L_{i-1}=L'_{i-1}$. For $i \in I$ we define $\E_i \1_\uk$ and $\1_\uk \F_i$ using kernels 
\begin{align*}
\O_{W_i^1(\uk)} \otimes \det(\sV'_i/\sV_{i-1}) \otimes \det(\sV_{i+1}/\sV_i)^{-1} \{k_i-1\} \in D(Y(\uk) \times Y(\uk+\alpha_i)) & \ \text{ and } \\
\O_{W_i^1(\uk)} \otimes \det(\sV'_i/\sV_i)^{k_{i+1}-k_i+1} \{k_{i+1}\} \in D(Y(\uk+\alpha_i) \times Y(\uk)) &
\end{align*}
respectively, while $\E_{i,1} \1_\uk$ and $\1_\uk \F_{i,-1}$ are defined by kernels
\begin{align*}
\O_{W_i^1(\uk)} \otimes \det(\sV'_i/\sV_{i-1}) \otimes \det(\sV_{i+1}/\sV'_i)^{-1} \{k_i-1+i\} \in D(Y(\uk) \times Y(\uk+\alpha_i)) & \ \text{ and } \\
\O_{W_i^1(\uk)} \otimes \det(\sV'_i/\sV_i)^{k_{i+1}-k_i+2} \{k_{i+1}-i\} \in D(Y(\uk+\alpha_i) \times Y(\uk)) &
\end{align*}
In this setup, since we are on the skew side, we have $\la 1 \ra = [1]\{-1\}$.  

The varieties $Y(\uk)$ have natural deformations $\tY(\uk) \rightarrow \bA^n$ defined by
\begin{align*}
\tY(\uk) \coloneqq \{(L_\bullet, \ux \in \C^n): & \C[z]^m = L_0 \subseteq L_1 \subseteq \dots \subseteq L_n \subseteq \C(z)^m \\ 
& (z-x_i) L_i \subset L_{i-1}, \dim(L_i/L_{i-1}) = k_i \}.
\end{align*}
\begin{Remark}
Note that, although $\tY(1^n) \cong \bY(1^n)$, in general $\tY(\uk)$ is just a subvariety of $\bY(\uk)$. 
\end{Remark}
Identifying the tangent bundle fiber $T_0(\bA^n)$ with the weight lattice of $GL_n$ we have $\alpha_i = (\dots,-1,1,\dots)$ where the $-1$ is in position $i$ and the rest of the entries are zero. Following the discussion in Appendix~\ref{sec:appdeform} this gives us a linear map 
\begin{align}\label{eq:gamma}
\varpi: \spn\{\alpha_i: i \in \hI \} \rightarrow \Hom(\O_{\Delta_{Y(\uk)}}, \O_{\Delta_{Y(\uk)}} [2]\{-2\}) = \End^2(\1_\uk).
\end{align}
The equality above is just by definition since $\la 2 \ra = [2]\{-2\}$.

\begin{Proposition}\label{prop:action}
The functors $\E_{i,\ell},\F_{i,-\ell}$ for $i \in I$ and $\ell \in \{0,1\}$ together with the map from \eqref{eq:gamma} define an $(L\gl_n,\t)$ action on $\K^{n,N}_{\Gr,m}$. 
\end{Proposition}
\begin{proof}
We refer to \cite[Section 4]{CK4} for a precise definition of an $(L\gl_n,\t)$ action. It is easy to see that the earlier work in \cite{CKL1,C2} proves all the conditions with the exception of (vi). This condition states that, for every $i,j \in I$ with $\la i,j \ra = -1$, one has $\Cone(\alpha) \cong \Cone(\beta)$ where $\alpha$ and $\beta$ are the unique (up to scaling) maps 
$$\E_i \E_{j,1} [-1]\{1\} \xrightarrow{\alpha} \E_{j,1} \E_i \ \ \text{ and } \ \ \E_j \E_{i,1} [-1]\{1\} \xrightarrow{\beta} \E_{i,1} \E_j.$$

We assume from now on that $j=i+1$ (there is no loss of generality in doing this). Let us denote by $\sE_i$ and $\sE_{j,1}$ the kernels of $\E_i$ and $\E_{j,1}$. To prove $\Cone(\alpha) \cong \Cone(\beta)$ we need to first understand the convolutions $\sE_i * \sE_{j,1}$ and $\sE_{j,1} * \sE_i$ more explicitly. Denote \begin{align*}
W_{ij}(\uk) &= \{(L_\bullet,L'_\bullet): L_k = L_k' \text{ for } k \ne i,j, L_i \subset L'_{i+1} \} \subset Y(\uk) \times Y(\uk+\alpha_i+\alpha_j) \\
W_{ji}(\uk) &= \{(L_\bullet,L'_\bullet): L_k = L_k' \text{ for } k \ne i,j, zL_{i+1} \subset L'_i \} \subset Y(\uk) \times Y(\uk+\alpha_i+\alpha_j). 
\end{align*}
To simplify notation we just write $W_{ij}$ and $W_{ji}$ for $W_{ij}(\uk)$ and $W_{ji}(\uk)$. It is not difficult to show (cf. \cite[Lemma 3.8]{CK2}) that 
\begin{align*}
\sE_i * \sE_{j,1} & \cong \O_{W_{ij}} \otimes \det(\sV_i'/\sV_{i-1}) \otimes \det(\sV_{i+2}/\sV'_{i+1})^{-1} \{s-1\} \\
\sE_{j,1} * \sE_i & \cong \O_{W_{ji}} \otimes \det(\sV_{i+1}/\sV'_{i+1})^{-1} \otimes \det(\sV'_{i+1}/\sV_{i-1}) \otimes \det(\sV_{i+2}/\sV_i)^{-1} \{s\}
\end{align*}
where $s := k_i + k_{i+1} - 1 + j$. This is because in the convolution diagram everything intersects in the expected dimension and so it is just a question of carefully keeping track of the line bundles (and $\{\cdot\}$ grading).

The next step is to understand the morphism $\alpha\colon \sE_i * \sE_{j,1} [-1] \{1\} \rightarrow \sE_{j,1} * \sE_i$. The geometric explanation of this morphism is the standard exact triangle
\begin{equation}\label{eq:3}
\O_{W_{ji}} (-W_{ij} \cap W_{ji}) \rightarrow \O_{W_{ji} \cup W_{ij}} \rightarrow \O_{W_{ij}}.
\end{equation}
Here we write $\O_{W_{ji}} (- W_{ij} \cap W_{ji})$ for the ideal sheaf of $W_{ij} \cap W_{ji} \subset W_{ji}$, which happens to be a divisor. This divisor, which is the locus inside $W_{ji}$ where $L_i \subset L_{i+1}$, is carved out as the vanishing locus of $\sV_i/\sV'_i \rightarrow \sV_{i+1}/\sV'_{i+1}$. Hence 
\begin{equation}\label{eq:6}
\O_{W_{ji}} (-W_{ij} \cap W_{ji}) \cong \O_{W_{ji}} \otimes \det(\sV_i/\sV'_i) \otimes \det(\sV_{i+1}/\sV'_{i+1})^{-1}. 
\end{equation}
Rewriting \eqref{eq:3} we get the following exact triangle
\begin{equation}\label{eq:4}
\O_{W_{ij}} [-1] \rightarrow \O_{W_{ji}} \otimes \det(\sV_i/\sV'_i) \otimes \det(\sV_{i+1}/\sV'_{i+1})^{-1} \rightarrow  \O_{W_{ji} \cup W_{ij}}. 
\end{equation}
Tensoring \eqref{eq:4} with the line bundle 
$$\sL \coloneqq \det(\sV_i'/\sV_{i-1}) \otimes \det(\sV_{i+2}/\sV'_{i+1})^{-1} \{s\}$$ 
and then simplifying and comparing with the expressions for $\sE_i * \sE_{j,1}$ and $\sE_{j,1} * \sE_i$ computed above, we arrive at the exact triangle
\begin{equation}\label{eq:5}
\sE_i * \sE_{j,1} [-1]\{1\} \rightarrow \sE_{j,1} * \sE_i \rightarrow \O_{W_{ji} \cup W_{ij}} \otimes \sL.
\end{equation}
On the other hand, we know that $\Hom(\sE_i * \sE_{j,1}[-1]\{1\}, \sE_{j,1} * \sE_i)$ is spanned by $\alpha$ (this is a formal calculation involving adjunctions, see for instance Lemma A.4 in the appendix of \cite{C3}). It follows that $\Cone(\alpha) \cong \O_{W_{ji} \cup W_{ij}} \otimes \sL$. 

One can similarly show that $\Cone(\beta) \cong \O_{W_{ji} \cup W_{ij}} \otimes \sL$. The main difference is that, instead of \eqref{eq:3}, one uses the standard exact triangle
$$\O_{W_{ij}}(-W_{ji} \cap W_{ij}) \rightarrow \O_{W_{ij} \cup W_{ji}} \rightarrow \O_{W_{ji}}$$
and instead of the isomorphism in \eqref{eq:6} one uses 
$$\O_{W_{ji}} (-W_{ij} \cap W_{ji}) \cong \O_{W_{ji}} \otimes \det(\sV_{i+1}/\sV'_{i+1}) \otimes \det(\sV_i/\sV'_i)^{-1} \{-2\}.$$
This is because $W_{ij} \cap W_{ji}$ inside $W_{ji}$ is the locus where $zL_{i+1} \subset L_i'$ and hence is carved as the vanishing of the morphism $z\colon \sV_{i+1}/\sV'_{i+1} \rightarrow \sV_i/\sV'_i \{2\}$. 
\end{proof}

As with the action on $\tK^{n,N}_{\Gr,m}$ we will define shifted generators $\E'_{i,1} \coloneqq \E_{i,1} \{-i\}$ and $\F'_{i,-1} \coloneqq \F_{i,-1} \{i\}$. The advantage is that it gets rid of the weird shift by $i$ in our original definition and simplifies subsequent shifts. The disadvantage is that $\Cone(\alpha)$ and $\Cone(\beta)$ are now only isomorphic up to a shift by $\{1\}$. Nevertheless, we can define $\T_0'$ and $\E_0,\F_0$ exactly as in \eqref{eq:T'0def} and \eqref{eq:E'0def} where again $\phi_i' = \T'_i \T_{i,1}'$.

Before we can prove the main result (Theorem \ref{thm:action2}) will need to identify $\phi'_i$ more explicitly (Corollary \ref{cor:phi}). 

\begin{Lemma}\label{lem:Ts}
Denote 
$$Z_i(1^N,1^N) \coloneqq \{(L_\bullet,L'_\bullet): L_j = L_j' \text{ if } j \ne i \} \subset Y(1^N) \times Y(1^N).$$
Then the kernels 
$$\O_{Z_i(1^N,1^N)} \{1\} \ \ \text{ and } \ \ \O_{Z_i(1^N,1^N)} \otimes \det(\sV_i/\sV_{i-1}) \otimes \det(\sV_{i+1}/\sV'_i)^{-1} \{1\}$$
induce the braid element $\T_i'$ and its inverse acting on $D(Y(1^N))$. Similarly, 
\begin{align*}
& \O_{Z_i(1^N,1^N)} \otimes \det(\sV_i/\sV_{i-1})^{-1} \otimes \det(\sV'_i/\sV_{i-1}) \{1\} \ \ \text{ and } \\ 
& \O_{Z_i(1^N,1^N)} \otimes \det(\sV'_i/\sV_{i-1}) \otimes \det(\sV_{i+1}/\sV'_i)^{-1} \{1\}
\end{align*}
induce the braid element $\T'_{i,1}$ and its inverse acting on $D(Y(1^N))$. 
\end{Lemma}
\begin{proof}
We prove the first isomorphism (cf. Proposition 7.1 of \cite{CK4}). The space $Z_i(1^N,1^N$) has two components: the diagonal $\Delta$ (where $L_i=L_i'$) and the locus $W_i$ where $zL_{i+1} \subset L_{i-1}$. This gives us the exact triangle
$$\O_{W_i}(-\Delta \cap W_i) \rightarrow \O_{Z_i(1^N,1^N)} \rightarrow \O_\Delta$$
of kernels in $Y(1^N) \times Y(1^N)$. On the other hand, $\Delta \cap W_i$ inside $W_i$ is the divisor consisting of points where $L_i=L_i'$ and hence is carved out by the vanishing locus of $\sV_i/\sV_{i-1} \rightarrow \sV_{i+1}/\sV'_i$. Hence we have the exact triangle
\begin{equation}\label{eq:7}
\O_{W_i} \otimes \det(\sV_i/\sV_{i-1}) \otimes \det(\sV_{i+1}/\sV'_i)^{-1} \rightarrow \O_{Z_i(1^N,1^N)} \rightarrow \O_\Delta.
\end{equation}
On the other hand, if we denote by $\sT_i, \sE_i, \sF_i$ the kernel inducing $\T_i,\E_i,\F_i$ then we have the exact triangle
$$\sT_i \rightarrow \O_\Delta \rightarrow \sF_i * \sE_i [1]\{-1\}.$$
It is straight forward to compute that 
$$\sE_i * \sF_i \cong \O_{W_i} \otimes \det(\sV_i/\sV_{i-1}) \otimes \det(\sV_{i+1}/\sV'_i)^{-1} \{1\}.$$
Thus we can rewrite the exact triangle above as 
$$\O_{W_i} \otimes \det(\sV_i/\sV_{i-1}) \otimes \det(\sV_{i+1}/\sV'_i)^{-1} \rightarrow \sT_i \rightarrow \O_\Delta.$$
Comparing it with \eqref{eq:7} we get that $\sT_i \cong \O_{Z_i(1^N,1^N)}$. Since $\T'_i = \T_i \{1\}$ this proves the first isomorphism. The other isomorphisms are proved similarly. 
\end{proof}

\begin{Corollary}\label{cor:phi}
The functor $\phi'_i \1_\uk$ is given by tensoring with the line bundle 
$$\det(\sV_i/\sV_{i-1})^{-1} \otimes \det(\sV_{i+1}/\sV_i).$$
\end{Corollary}
\begin{proof}
The claim when $\uk = (1^N)$ follows from Lemma \ref{lem:Ts}. The claim in general then follows as in the proof of \cite[Corollary 7.3]{CK4} (where one proves the analogous result on the symmetric side). 
\end{proof}

We now have the following result parallel to Theorem \ref{thm:action1}. 

\begin{Theorem}\label{thm:action2}
For $i \in \hI$ the functors $\E_i,\F_i$ defined above together with the composition $\spn\{\alpha_i: i \in \hI \} \xrightarrow{p} \spn\{\alpha_i: i \in I \} \rightarrow \End^2(\1_\uk)$ induce an $(L\gl_n,\t)_{KM}$ action on $\tK_{\Gr,m}^{n,N}$. Here $p$ is the map from (\ref{p:map}) while the second map is the one from \eqref{eq:gamma}.
\end{Theorem}
\begin{proof}
This follows from the exact same argument used to prove Theorem \ref{thm:action1}, where we just use Lemma \ref{lem:Taction2} in place of Lemma \ref{lem:Taction1}.
\end{proof}

\begin{Lemma}\label{lem:Taction2}
Working with $\K^{n,N}_{\Gr,m}$ we have the commutative diagrams
\begin{align*}
\xymatrix{
\spn\{\alpha_i: i \in \hI \} \ar[d]^{\varpi} \ar[r]^{T_i} & \spn\{\alpha_i: i \in \hI \} \ar[d]^{\varpi} & \spn\{\alpha_i: i \in \hI \} \ar[d]^\varpi \ar[r]^{id} & \spn\{\alpha_i: i \in \hI \} \ar[d]^\varpi \\
\End^2(\1_\uk) \ar[r]^{T'_i} & \End^2(\1_{s_i \cdot \uk}) & \End^2(\1_\uk) \ar[r]^{\phi'_i} & \End^2(\1_{\uk}) }
\end{align*}
where the map $T_i \in \End(\spn\{\alpha_i: i \in \hI \})$ is the standard action given by reflecting in $\alpha_i$ and $T'_i,\phi_i'$ are the isomorphisms induced by the equivalences $\T'_i, \phi_i'$.
\end{Lemma}
\begin{proof}
The fact that $\phi'_i$ acts trivially is a consequence of Corollary \ref{cor:phi}. 

Now denote by $\sT'_i$ the kernel for $\T'_i$. Then, for $\theta \in \spn\{\alpha_i: i \in \hI \}$, the 2-morphism
$$-T'_i(\varpi(\theta)) * id * id + id * id * \varpi(\theta) \in \End^2(\1_{s_i \cdot \uk} * \sT_i * \1_\uk)$$
is equal to zero, since $T'_i(\varpi(\theta)) * id * id = id * id * \varpi(\theta)$ by the definition of $T'_i$. 

On the other hand, it follows from \cite{CK3} that the kernel $\sT_i \in D(Y(s_i \cdot \uk) \times Y(\uk))$ deforms over $(\tY(\uk) \times \tY(s_i \cdot \uk))_{(T_i(\theta),\theta)} \rightarrow \Spec (\C[\vareps]/\vareps^2)$. More precisely, Theorem A.10 of \cite{CK3} shows that $\sT_i$ deforms to $\tY(\uk) \times \tY(s_i \cdot \uk)$ restricted to the locus 
$$(x_1,\dots,x_i,x_{i+1},\dots,x_n) \times (x_1, \dots,x_{i+1},x_i,\dots,x_n) \subset \bA^n \times \bA^n.$$ 
This implies the statement above. Subsequently, the discussion in Appendix \ref{sec:appdeform} implies that 
$$-\varpi(T_i(\theta)) * id * id + id * id * \varpi(\theta) \in \End^2(\1_{s_i \cdot \uk} * \sT_i * \1_\uk)$$
is equal to zero. it follows that $-T'_i(\varpi(\theta)) * id * id = -\varpi(T_i(\theta)) * id * id$ and hence $T'_i(\varpi(\theta)) = \varpi(T_i(\theta))$. 
\end{proof}

\section{Induced t-structures}\label{sec:t}

\subsection{Perverse coherent sheaves}\label{sec:pcoh}

In order to apply Theorem~\ref{thm:sym:main}, we need to start with a t-structure on $\tK(\eta)$. It turns out that for our applications perverse coherent t-structures are a good choice. Let us briefly recall their definition from \cite{AB}. In this subsection, $X$ will be a scheme with an action of an algebraic group $H$. 

We denote by $D(X)$ the bounded derived category of $H$-equivariant coherent sheaves on $X$ ($H$ could be trivial) and by $\Dqcoh{X}$ the corresponding unbounded category of $H$-equivariant quasi-coherent sheaves on $X$. 

Let $X^{\gen}$ be the topological space consisting of generic points of $H$-stable subschemes of $X$. For $x \in X^{\gen}$ and $\sF \in D(X)$ we write $\bi_x^* \sF$ for the usual stalk of $\sF$ at $x$ and $\bi_x^! \sF = \bi_x^* \Gamma_{\overline{x}} \sF$ the stalk of the local cohomology along the closure of $\{x\}$. We note that these are functors of sheaves of $\C$-modules, but not of $\O$-modules.

Choose a perversity function $p\colon X^{\gen} \to \bZ$ and define 
\begin{align*}
    {}^p\!D^{\le 0}(X) & = \bigl\{ \sF \in D(X) : \bi_x^* \sF \in D_{qcoh}^{\le p(x)}(\O_x) \text{ for all } x \in X^{\gen} \bigr\} \\
    {}^p\!D^{\ge 0}(X) & = \bigl\{ \sF \in D(X) : \bi_x^! \sF \in D_{qcoh}^{\ge p(x)}(\O_x) \text{ for all } x \in X^{\gen} \bigr\}.
\end{align*}

We say a perversity function $p$ is monotone if $p(x) \le p(y)$ whenever $y$ is contained in the closure of $x$. We say $p$ is comonotone if $\overline{p}(x) = - \dim(x) - p(x)$ is monotone. 

\begin{Theorem}\cite[Thm. 3.10]{AB}\label{thm:perverse}
If $p\colon X^\gen \rightarrow \bZ$ is monotone and comonotone then 
$$({}^p\!D^{\le 0}(X),\, {}^p\!D^{\ge 0}(X))$$
defines a t-structure on $D(X)$, called the perverse coherent t-structure. 
\end{Theorem}

Common choices for perversity functions include the standard perversity $p(x) = 0$ (yielding the standard t-structure), the dual standard perversity $p(x) = -\dim(x)$ and the middle perversity $p(x) = - \frac{1}{2} \dim(x) $. Strictly speaking the middle perversity function is only defined if all $\dim(x)$ are even. However, we will allow $\dim(x)$ to be odd by rounding $- \frac{1}{2} \dim(x)$ either up or down. This leads to a family of functions (depending on the rounding choices we make). Abusing terminology slightly we will still refer to a function in this family as a middle perversity function. 

From the definition above it follows easily that the perverse coherent t-structure on $D(X)$ is ``local'' in the sense of \cite[Section~2.3]{Po}. In particular this means that tensoring with any vector bundle is a t-exact functor. This is one of the main properties that will be useful in the next section. Another one is the following.

\begin{Lemma}\cite[Lemma 3.3]{AB}\label{lem:pcoh}
Consider a perversity function $p\colon X^\gen \rightarrow \bZ$ and for a closed subscheme $i\colon Y \hookrightarrow X$ denote by $p\colon Y^\gen \rightarrow \bZ$ the induced perversity function given by composition. Then $i_*\colon D(Y) \rightarrow D(X)$ is t-exact with respect to the perverse coherent t-structures on $D(Y)$ and $D(X)$. 
\end{Lemma}

We call the t-structure on $D(Y)$ corresponding to the pullback of the perversity function $p\colon X^\gen \rightarrow \bZ$ the ``$i$-induced perverse coherent t-structure''.

\subsection{Induced t-structures on $D(\bY(\uk))$}\label{sec:induced1}

We begin with a couple of technical results. Recall that, as in Section \ref{sec:notation}, we have fixed a subgroup $H \subseteq GL_m[[z]]$ and all categories $D(\bY(\uk))$ are $H \rtimes \C^\times$ equivariant. 

\begin{Lemma}\label{lem:A}
For $k,l \in \N$ consider the projection 
$$\pi\colon \bY(k,l) \rightarrow \bY(k+l) = \{L_0 \subset L_2: zL_2 \subset L_2, \dim(L_2/L_0)=k+l \}$$ 
which forgets $L_1$. Then $\pi_* \det(\sV_1)^{-1} \cong \Lambda^k(\sV_2)^{-1}$ and $\pi_* \det(\sV_1) \cong \Lambda^k(\sV_2) \{-2kl\}$. 
\end{Lemma}
\begin{proof}
Consider the embedding $\bY(k,l) \xrightarrow{i} \G(k,\sV_2) \xrightarrow{p} \bY(k+l)$ where $\G(k,\sV_2)$ denotes the relative Grassmannian of $k$-planes in the $(k+l)$-dimensional vector bundle $\sV_2$ on $\bY(k+l)$. Notice that the codimension of $i$ is $kl$. On the other hand, we have the natural map of vector bundles on $\G(k,\sV_2)$
$$z\colon \sV_1 \rightarrow \sV_2/\sV_1 \{2\}$$
where, abusing notation slightly, $\sV_1$ denotes the tautological $k$-dimensional bundle on $\G(k,\sV_2)$. This map vanishes precisely over the image of $i$ (i.e. the locus where $z$ maps $\sV_1$ to itself). In particular, we get a global section of $(\sV_2/\sV_1) \otimes \sV_1^\vee \{2\}$ which vanishes along the image of $i$. Notice that the dimension of this vector bundle is equal to $kl$ which is the codimension of $i$. Moreover, one can check that generically the vanishing locus of this section is reduced. It follows that we get a Koszul resolution of $i_* \O_{\bY(k,l)}$ given by 
$$\dots \rightarrow \Omega_p^2 \{-4\} \rightarrow \Omega_p^1 \{-2\} \rightarrow \O_{\G(k,\sV_2)} \rightarrow i_* \O_{\bY(k,l)}$$
where $\Omega_p^k = \Lambda^k \Omega_p^1$ and $\Omega_p^1 = \sV_1 \otimes (\sV_2/\sV_1)^\vee$ is the relative cotangent bundle of $p$. Tensoring everything with $\det(\sV_1)^{-1}$ and using that 
$$H^p(\G(k,k+l), \Omega_{\G(k,k+l)}^q(1)) = 0 \ \text{ if } (p,q) \ne (0,0)$$
(see \cite[Corollaire 1]{LP}) we get that 
$$\pi_* \det(\sV_1)^{-1} \cong p_* i_* \O_{\bY(k,l)} \cong R^0p_* \left( \O_{\G(k,\sV_2)} \otimes \det(\sV_1)^{-1} \right) \cong \Lambda^k(\sV_2)^\vee.$$

Applying Verdier duality to this result we get 
\begin{align*}
\pi_* \det(\sV_1) 
&\cong \pi_* \D (\det(\sV_1)^{-1} \otimes \omega_{\bY(k,l)}) [-\dim \bY(k,l)] \\
&\cong \D \pi_* (\det(\sV_1)^{-1} \otimes \omega_\pi \otimes \pi^* \omega_{\bY(k+l)}) [-\dim \bY(k,l)] \\
&\cong \D (\pi_* (\det(\sV_1)^{-1} \{2kl\}) \otimes \omega_{\bY(k+l)}) [-\dim \bY(k+l)] \\
&\cong \D (\Lambda^k (\sV_2)^\vee) \{-2kl\} \otimes \omega_{\bY(k+l)}^{-1} [-\dim \bY(k+l)] \\ 
&\cong \Lambda^k (\sV_2) \{-2kl\}
\end{align*}
where we used \cite[Lemma 5.2]{CK4} to determine $\omega_\pi$ in the third line. This completes the proof. 
\end{proof}

\begin{Corollary}\label{cor:A}
For $0 \le \ell \le N$ the endomorphism $\A^{(\ell)}$ of $D(\bY(\eta))$ is given by tensoring with $\Lambda^\ell(\sV_n)$ while $\A^{(-\ell)}$ is tensoring with $\Lambda^\ell(\sV_n)^\vee$. 
\end{Corollary}
\begin{proof}
Suppose $\ell \ge 0$. It suffices to consider the case $n=2$ where we have $\bY(0,N)$ and $\A^{(\ell)}$ is the composition 
$$D(\bY(0,N)) \xrightarrow{\F_1^{(\ell)}} D(\bY(\ell,N-\ell)) \xrightarrow{\F_0^{(\ell)}} D(\bY(0,N)).$$
Now $\F_0^{(\ell)} \cong \T'_1 (\F'_{1,-1})^{\ell} (\T'_1)^{-1}$ so we get 
\begin{align*}
\F_0^{(\ell)} \F_1^{(\ell)} \1_{(0,N)} 
&\cong \T'_1 (\F'_{1,-1})^{(\ell)} (\T'_1)^{-1} \F_1^{(\ell)} \1_{(0,N)} \\
&\cong \T'_1 (\F'_{1,-1})^{(\ell)} \E_1^{(\ell)} (\T'_1)^{-1} \1_{(0,N)} \{-\ell(N-\ell)\}
\end{align*}
where we used \eqref{eq:TE=FT} to get the second isomorphism. Now consider the composition $(\F'_{1,-1})^{(\ell)} \E_1^{(\ell)}$ in the middle. Following the definition from \cite[Section 5]{CK4} this is the functor
$$\pi_*(\pi^*(\cdot) \otimes \det(\sV_2/\sV_1) \{\ell(N-\ell)\})$$
where $\pi\colon \bY(N-\ell,\ell) \rightarrow \bY(N,0)$ is the natural projection. Applying the projection formula and using Lemma \ref{lem:A} this equivalent to tensoring by 
$$\Lambda^{N-\ell}(\sV_2)^\vee \otimes \det(\sV_2) \{\ell(N-\ell)\} \cong \Lambda^\ell(\sV_2) \{\ell(N-\ell)\}.$$ 
The result follows. The case of $\ell \le 0$ is obtained by taking adjoints and noting that
$(\F_0^{(\ell)} \F_1^{(\ell)})^R \cong \E_1^{(\ell)} \E_0^{(\ell)}$. 
\end{proof}

\begin{Theorem}\label{thm:sym-app}
Fix a perverse coherent t-structure on $\bY(\eta)$. Then there exists a unique t-structure on $\oplus_\uk D(\bY(\uk))$ which agrees with the chosen t-structure on $D(\bY(\eta))$ and such that $\E_i,\F_i$ are t-exact for all $i \in \hI$. Moreover, this t-structure is braid positive. 
\end{Theorem}
\begin{proof}
We apply Theorem \ref{thm:sym:main}. By Corollary \ref{cor:A} we know that $\A^{(\ell)}$ is t-exact for $-N \le \ell \le N$. So it remains to show that $D(\bY(\mu))$ is weakly generated by the images of $D(\bY(\eta))$ under the $L\gl_n$ action. 

Consider the map $\ch\colon \bY(\mu) = \bY(1^N) \rightarrow \bA^N$ which remembers the eigenvalues of $z$ acting on $L_i/L_{i-1}$ for $i=1,\dots,N$ (in this case all of these have rank one). The fibers of this map are iterated $\P^{m-1}$-bundles. For example, the fiber over $0 \in \bA^N$ is isomorphic to $Y(1^N)$ while the fiber over a general point (where all the coordinates are different) is just a product $(\P^{m-1})^N$. In particular, this means that $D(\bY(\mu))$ is weakly generated by line bundles of the form $\prod_i \sL_i^{a_i}$ where $\sL_i \coloneqq \det(\sV_i/\sV_{i-1})$ and $\sum_i a_i \equiv 0 \pmod N$.  

On the other hand, consider the endomorphism $\phi'_i \1_\mu = \T'_i \T_{i,1}' \1_\mu$ of $D(\bY(\mu))$ (see \cite{CK4} or the appendix for details). By \cite[Corollary 7.3]{CK4} it is given by tensoring with the line bundle $\sL_i^{-1} \otimes \sL_{i+1}$ (N.B. Remark \ref{rem:convention}). Moreover, a planar tree morphism induces a map $D(\bY(\eta)) \rightarrow D(\bY(\mu))$ which is just pullback via the natural projection $\bY(1^N) \rightarrow \bY(N)$. Thus if we apply this map to the line bundle $\det(\sV_1/\sV_0)$ on $\bY(N)$ we get the line bundle $\prod_i \sL_i$ on $\bY(1^N)$. If we then apply various $\phi'_i$ we can obtain all line bundles of the form $\prod_i \sL_i^{a_i}$ with $\sum_i a_i \equiv 0 \pmod N$. This shows weak generation of $D(\bY(\mu))$ and completes the proof. 
\end{proof}

\subsection{Induced t-structures on $D(Y(\uk))$}\label{sec:induced2}

We will now use Corollary \ref{cor:skewsym} to define a t-structure on the $D(Y(\uk))$. We take $\mu = (1^N)$ and consider the fiber product diagram
$$\xymatrix{
Y(\mu) \ar[r]^\iota \ar[d] & \bY(\mu) \ar[d]^{\ch} \\
0 \ar[r] & \bA^N }$$

Proposition 8.1 of \cite{CK4} implies that $\T'_i \circ \iota_* \cong \iota_* \circ \T'_i$ for all $i \in I$. Moreover, \cite[Corollary 7.3]{CK4} and Corollary \ref{cor:phi} tell us that $\phi'_i$ acting on $D(\bY(\mu))$ and $D(Y(\mu))$ is given by tensoring with the same line bundle (N.B. Remark \ref{rem:convention}). This means that $\phi'_i \circ \iota_* \cong \iota_* \circ \phi'_i$ and hence we also have $\T'_0 \circ \iota_* \cong \iota_* \circ \T'_0$. 

Now $\iota_*$ is conservative and its left adjoint is $\iota^*$. So it remains to check that $\iota_* \circ \iota^*\colon D(\bY(\mu)) \rightarrow D(\bY(\mu))$ is right t-exact. To see this note that $\iota_* \circ \iota^*$ is given by tensoring with $\iota_* \O_{Y(\mu)}$. Moreover, $Y(\mu) \subset \bY(\mu)$ is cut out by $x_1, \dots, x_N$ where $\bA^N = \Spec \C[x_1, \dots, x_N]$. So we have a Koszul resolution  
$$\bigotimes_i \left( \O_{\bY(\mu)} \{-2\} \xrightarrow{x_i} \O_{\bY(\mu)} \right)$$
of $\iota_* \O_{Y(\mu)}$. Since tensoring with $\O_{\bY(\mu)}$ is the identity (and hence t-exact) this shows that tensoring with $\iota_* \O_{Y(\mu)}$ is right t-exact.

The following is now a consequence of Corollary \ref{cor:skewsym}.

\begin{Theorem}\label{thm:skew-app}
Fix a perverse coherent t-structure on $\bY(\eta)$. Then there exists a unique t-structure on $\oplus_\uk D(Y(\uk))$ such that 
\begin{enumerate}
\item $\iota_*\colon D(Y(\mu)) \rightarrow D(\bY(\mu))$ is t-exact with respect to the t-structure on $D(\bY(\mu))$ induced as in Theorem \ref{thm:sym-app} and
\item the 1-morphism $D(Y(\uk)) \rightarrow D(Y(\mu))$ induced by a collection of planar trees is t-exact. 
\end{enumerate}
Moreover, this t-structure on $\oplus_\uk D(Y(\uk))$ is braid positive.
\end{Theorem}

Just like we have a natural projection map 
$$\tpi\colon \bY(\mu) \rightarrow \bY(\eta) = \{\C[z]^m = L_0 \subseteq L_n \subseteq \C(z)^m : z L_n \subset L_n, \dim(L_n/L_0) = N \}$$ 
which forgets $L_1,\dots,L_{n-1}$ we also have a natural map $\pi\colon Y(\mu) \rightarrow \oY(\mu)$ where 
$$\oY(\mu) = \{ \C[z]^m = L_0 \subseteq L_n \subseteq \C(z)^m : z^n L_n \subset L_0, \dim(L_n/L_0) = N \}.$$
This gives us a commutative diagram 
$$\xymatrix{
Y(\mu) \ar[r]^\iota \ar[d]^{\pi} & \bY(\mu) \ar[d]^{\tpi} \\
\oY(\mu) \ar[r]^i & \bY(\eta) }$$
where $i\colon \oY(\mu) \rightarrow \bY(\eta)$ is the natural inclusion. 
\begin{Remark}
In the above we really want $\oY(\mu)$ and not $Y(\eta)$ or $\oY(\eta)$ which are in general smaller (sometimes empty) varieties. Using the notation from Section \ref{sec:conv-vars}, the varieties $Y(\mu)$ and $\oY(\mu)$ are identified with $\Gr^{(\omega_1, \dots, \omega_1)}$ and $\Gr^{N \omega_1}$ respectively while the map $\pi$ corresponds to $m$. 
\end{Remark}

\begin{Corollary}\label{cor:skew-app}
The functor $\pi\colon D(Y(\mu)) \rightarrow D(\oY(\mu))$ is t-exact with respect to the t-structure on $D(Y(\mu))$ from Theorem \ref{thm:skew-app} and the $i$-induced perverse coherent t-structure on $D(\oY(\mu))$. 
\end{Corollary}

\begin{proof}
We know that $\iota_*$ is t-exact and that $\tpi_*$ is t-exact since it is isomorphic (up to some grading shift $\{\cdot\}$) to a composition of $\E_i$s for various $i \in I$ (a planar tree). This means that the composition $\tpi_* \circ \iota_* \cong i_* \circ \pi_*$ is t-exact. But $i_*$ is t-exact by Lemma \ref{lem:pcoh} which then implies that $\pi_*$ must also be t-exact. 
\end{proof}

\section{Some final remarks}\label{sec:finalrem}

Recall that in \cite{BM} one considers two types of non-standard t-structures, called the \enquote{exotic} and \enquote{perversely exotic} t-structures.

The exotic t-structures are defined on the \emph{non-equivariant} derived categories $D(\tsN)$ and $D(\tg)$ of coherent sheaves on the Springer resolution and Grothendieck--Springer resolution respectively.
They are axiomatically defined by the requirement that they are braid positive and that the pushforward along the projection map $\pi\colon \tsN \to \sN$ (resp.~$\tpi\colon \tg \to \g$) is t-exact with respect to the standard t-structure on the base.

On the other hand, the perversely exotic t-structure is defined on the derived category $D^G(\tsN)$ of \emph{equivariant} coherent sheaves on the Springer resolution.  Due to the fact that $D^G(\tsN)$ is not finitely generated, its axiomatic description will need an additional condition, which we will describe next. 
 
One says that a t-structure on a triangulated category $\sC$ is compatible with a thick triangulated subcategory $\sC'$ if there exist t-structures on $\sC'$ and $\sC/\sC'$ such that the inclusion and projection maps are t-exact.  One extends this definition in the obvious way to define compatibility of a t-structure with a filtration by thick triangulated subcategories.  The support filtration on $D^G(\tsN)$ is then the filtration given by the full subcategories consisting of complexes supported on the preimage of each $G$-orbit closure in $\sN$. 
 
The perversely exotic t-structure is the (unique) t-structure on $\tsN$, which is
\begin{itemize}
    \item compatible with the support filtration,
    \item braid positive, and
    \item such that the functor $\pi_*\colon D^G(\tsN) \to D^G(\sN)$ is t-exact with respect to the perverse coherent t-structure of middle perversity on the base.
\end{itemize}
We will now discuss how these t-structures arise from our considerations.

\subsection{Grothendieck--Springer resolution}\label{sec:springer1}

We first focus on the 2-category $\tK^{n,n}_{\Gr,n}$ (i.e. the special case $m=n=N$) where we take $H \subseteq GL_n \subset GL_n[[z]]$. 

In this case $\bY(\eta) \cong \bY(n)$. A special case of Theorem 3.1 in \cite{CK4} shows that there is an open subset $\bY(n)^o \subset \bY(n)$ which is isomorphic to $\g \coloneqq \gl_n$ (the space of $n \times n$ matrices). This isomorphism, which is $GL_n$-equivariant but not $GL_n[[z]]$-equivariant (this is why we restricted $H$ to $GL_n$), is obtained in two steps. First one identifies $\bY(n)$ with 
$$\bX(n) = \{\C[z]^n = \C[z] \otimes_\C V = L_0 \supset L: zL \subset L, \dim(L_0/L)=n \}$$
where $V = \spn \{e_1, \dots, e_n\}$. Then one considers the open locus $\bX(n)^o$ consisting of $L$ such that the natural projection $L_0/L \rightarrow V$ is an isomorphism (this projection is induced by $\C[z] \rightarrow \C$ where $z \mapsto 0$). Equivalently this means that $\{e_1, \dots, e_n\}$ forms a basis of $L_0/L$. This construction is an analogue of the Mirkovi\'c--Vybornov isomorphism \cite{MVy}.

More generally, one can identify $\bY(\uk)$ with 
$$\bX(\uk) = \{\C[z]^n = L_0 \supset L_1 \supset \dots \supset L_n, zL_i \subset L_i, \dim(L_{i-1}/L_i) = k_i \}$$ 
and then consider the open subset $\bX(\uk)^o \subset \bX(\uk)$ where $L_0/L_n \xrightarrow{\sim} V$ (notice that we are assuming $\sum_i k_i = n$). Translating the $L_i$ under this isomorphism allows us to identify $\bX(\uk)^o$ with 
$$\{(X,V_\bullet): X \in \gl_m, 0=V_0 \subset V_1 \subset \dots \subset V_n = \C^n, XV_i \subset V_i, \dim(V_{i-1}/V_i)=k_i\}.$$
Let us denote this space $\tg_\uk$ (in the notation of \cite{BM} it is denoted $\tg_\sP$). The space $\tg_{(1^n)}$ will be denoted just $\tg$. The natural map $\tpi\colon \tg \rightarrow \g$ which forgets $V_1,\dots,V_{n-1}$ is just the Grothendieck--Springer resolution. 

We can form a 2-category $\tK_\g$ out of $\oplus_\uk D(\tg_\uk)$. It is easy to see that the $(L\gl_n,\t)$ action on $\tK^{n,n}_{\Gr,n}$ restricts to give an action on $\tK_\g$. 

\begin{Lemma}\label{lem:braidagrees1}
The induced affine braid group action on $D(\tg)$ generated by $\T'_i, \phi'_i$ for $i \in I$ agrees with the action from \cite[Theorem 1.6.1]{BR}. 
\end{Lemma}
\begin{Remark}
\cite{BR} actually consider an action of the \emph{extended} affine braid group, while here we only consider the non-extended version. This does not affect the braid positivity condition.
\end{Remark}
\begin{proof}
The generators $\T'_i$ agree by \cite[Proposition 7.1]{CK4} while the generators $\phi'_i$ agree by \cite[Corollary 7.3]{CK4} (N.B. Remark \ref{rem:convention}).  
\end{proof}

\begin{Corollary}\label{cor:tg-app}
Fix a perverse coherent t-structure on $D(\g) = D(\tg_\eta)$. Then there exists a unique t-structure on $\oplus_\uk D(\tg_\uk)$ which agrees with the chosen t-structure on $D(\g)$ and such that $\E_i,\F_i$ are t-exact for all $i \in \hI$. Moreover, this t-structure is braid positive. Using the standard t-structure on $D(\g)$ (with $H$ trivial) recovers the exotic t-structure on $D(\tg)$ from \cite[Theorem 1.5.1]{BM}. 
\end{Corollary}
\begin{proof}
The first part follows exactly as in the proof of Theorem \ref{thm:sym-app}. To see that we recover the t-structure from \cite[Theorem~1.5.1]{BM} it suffices to show that $\pi_*\colon D(\tg) \rightarrow D(\g)$ is t-exact where $\pi\colon \tg \rightarrow \g$ is the standard projection. But $\pi_*$ is, up to a grading shift $\{\cdot\}$, isomorphic to the map corresponding to a planar tree between $\mu = (1^n)$ and $\eta$. Thus it is t-exact by the first part. 
\end{proof}

\subsection{Springer resolution}\label{sec:springer2}

Section \ref{sec:springer1} above was on the symmetric side of the picture. One has a similar story on the skew side. More precisely, we focus on the 2-category $\K^{n,n}_{\Gr,n}$ where again we take $H \subseteq GL_n \subseteq GL_n[[z]]$. 

Following \cite{MVy} each $Y(\uk)$ in this case has an open subset $\g_\uk$ isomorphic to the cotangent bundle to the partial flag variety
$$\{(X,V_\bullet): X \in \gl_m, 0=V_0 \subset V_1 \subset \dots \subset V_n = \C^n, XV_i \subset V_{i-1}, \dim(V_{i-1}/V_i)=k_i\}.$$
The space $\g_{(1^n)}$ will be denoted $\tsN$ and the natural map $\pi\colon \tsN \rightarrow \sN$ which forgets $V_1, \dots, V_{n-1}$ is just the Springer resolution of the nilpotent cone $\sN$. 

We can again form a 2-category $\K_\g$ out of $\oplus_\uk D(\g_{\uk})$ and the $(L\gl_n,\t)$ action on $\K^{n,n}_{\Gr,n}$ restricts to give an action on $\K_\g$. 

\begin{Lemma}\label{lem:braidagrees2}
The induced affine braid group action on $D(\tsN)$ generated by $\T'_i, \phi'_i$ for $i \in I$ agrees with the action from \cite[Theorem 1.6.1]{BR}.
\end{Lemma}
\begin{proof}
This can be seen either by comparing kernels as in the proof of Lemma \ref{lem:braidagrees1} or by using the result in Lemma \ref{lem:braidagrees1} and the fact that the braid group actions on $D(\tg)$ and $D(\tsN)$ are compatible via the natural inclusion $\iota\colon \tsN \rightarrow \tg$. 
\end{proof}

\begin{Corollary}\label{cor:g-app}
Fix a perverse coherent t-structure on $D(\g) = D(\tg_\eta)$. Then there exists a unique t-structure on $\oplus_\uk D(\g_\uk)$ such that
\begin{enumerate}
\item $\iota_*\colon D(\tsN) \rightarrow D(\tg)$ is t-exact with respect to the t-structure on $D(\tg)$ induced as in Corollary \ref{cor:tg-app} and 
\item the 1-morphism $D(\g_\uk) \rightarrow D(\tsN)$ induced by a collection of planar trees is t-exact.
\end{enumerate}
Moreover, this t-structure is braid positive on $\oplus_\uk D(\g_\uk)$. 

If we take the standard perversity function $p = 0$ on $\g$ and $H$ trivial then we recover the ``representation theoretic exotic t-structure'' on $D(\tsN)$ from \cite[Theorem 1.5.1]{BM}. 

If we take a middle perversity function on $\g$ and $H=G$ then we recover the perversely exotic t-structure on $D^G(\tsN)$ from \cite[Theorem 6.2.1]{BM}. 
\end{Corollary}
\begin{proof}
The first part follows exactly like in the proof of Theorem \ref{thm:skew-app}. When $p=0$ we recover the t-structure from \cite[Theorem 4.2]{BM} because our t-structure is braid positive and $\pi_*\colon D(\tsN) \rightarrow D(\sN)$ is t-exact by Corollary~\ref{cor:skew-app}. 

If we choose $H=G$ and a middle perversity function on $\g$ then the induced perversity function on $\sN \subset \g$ is also the middle perversity function. It follows again by Corollary~\ref{cor:skew-app} that $\pi_*\colon D^G(\tsN) \rightarrow D^G(\sN)$ is t-exact with respect to these t-structures. 

Finally, to check compatibility with the support filtration, fix an orbit closure $O \subseteq \sN \subset \g$. We can repeat our construction with the categories $D^G_{\tpi^{-1}O}(\tg_\uk)$ of complexes supported in the preimage of $O$ to obtain braid positive t-structures on these categories, such that the inclusion functors into $D^G(\tg_\uk)$ are t-exact.
Correspondingly we also get a t-structure on $D^G_{\pi^{-1}O}(\tsN)$ such that the inclusion functor is t-exact.
More generally, if $O_1 \subseteq O_2$ is an inclusion of orbit closures, then we get a compatible t-structure on $D^G_{\pi^{-1}O_2}(\tsN) / D^G_{\pi^{-1}O_1}(\tsN) \cong D^G_{\pi^{-1}(O_2\setminus O_1)}(\tsN \setminus \pi^{-1}O_1)$. Thus our t-structure on $D^G(\tsN)$ is compatible with the support filtration.  
\end{proof}

\appendix
\section{From the Drinfeld--Jimbo to the Kac--Moody presentation}\label{sec:appdrinfeld}

The concept of an $(L\gl_n,\t)$ action discussed in Section \ref{sec:action} involves the quantum affine algebra in its Kac--Moody presentation. On the other hand, in \cite{CK4} we constructed an action of $(L\gl_n,\t)$ in its Drinfeld--Jimbo presentation on $\oplus_\uk D(\bY(\uk))$. In this section we explain how these two are related. Note that we do not prove that an $(L\gl_n,\t)_{KM}$ action is equivalent to an $(L\gl_n,\t)_{DJ}$ action although something like this is almost certainly true. 

Briefly recall that an $(L\gl_n,\t)_{DJ}$ action uses generating 1-morphisms $\E_{i,1}$ and $\F_{i,-1}$ instead of $\E_0$ and $\F_0$ (for a complete definition see \cite[Section 4]{CK4}). 

Starting with an $(L\gl_n, \t)_{DJ}$ action we can define an affine braid group action generated by $\T_i$ and $\phi_i$ for $i=1,\dots,n-1$ (see \cite[Section 4.2]{CK4}). More precisely, $\T_i$ is defined as in Section \ref{sec:braidaction} using the $\sl_2$ pair $\E_i,\F_i$ while $\T_{i,1}$ is defined in the same way using the $\sl_2$ pair $\E_{i,1}, \F_{i,-1}$. Then $\phi_i$ is defined as $\T_i \T_{i,1}$.  

\begin{Remark}
Because of conventions (see Remark \ref{rem:convention}) $\phi_i$ is defined in \cite{CK4} as $\T_{i,1} \T_i$ rather than $\T_i \T_{i,1}$. For this same reason, \cite[Lemma 4.8]{CK4} which we will use, also takes a slightly different form as we now explain. For $\alpha = \sum_{i=1}^{n-1} a_i \alpha_i$ we write $\phi_\alpha \coloneqq \prod_i \phi_i^{a_i}$. Then \cite[Lemma 4.8]{CK4} implies that $\E_i \phi_\alpha \cong \phi_\alpha \E_i$ if $\la \alpha,\alpha_i \ra = 0$ and $\E_{i,1} \phi_\alpha \cong \phi_\alpha \E_i$ if $\la \alpha,\alpha_i \ra = -1$. 

Likewise, \cite[Lemma 4.4]{CK4} now states $\T_i^{-1} \E_{j,1} \T_i \cong \T_j^{-1} \E_{i,1} \T_j$ if $\la i,j \ra = -1$. Taking adjoints also gives us the analogous isomorphisms for $\F$s instead of $\E$s.
\end{Remark}

We now explain how to obtain $\E_0$ and $\F_0$ starting with an $(L\gl_n,\t)_{DJ}$ action. First we rewrite the affine braid group action in its Kac--Moody presentation. More precisely, we define 
\begin{equation}\label{eq:T0def}
\T_0 \coloneqq \phi_1 \dots \phi_{n-1} \T_1^{-1} \T_2^{-1} \dots \T_{n-1}^{-1} \dots \T_2^{-1} \T_1^{-1}.
\end{equation}
We now have $\T_i \T_j \T_i \cong \T_j \T_i \T_j$ whenever $i-j = \pm 1 \pmod n$. We then define
\begin{equation}\label{eq:E0def}
\E_0 \coloneqq \T_1^{-1} \T_0^{-1} \E_1 \T_0 \T_1 \ \  \text{ and } \ \ \F_0 \coloneqq \T_1^{-1} \T_0^{-1} \F_1 \T_0 \T_1.
\end{equation}
This definition is partly justified in light of relation \eqref{eq:TiTjEi}.

\begin{Lemma}\label{lem:app1}
We have $\E_0 \cong \T_{n-1}^{-1} \T_0^{-1} \E_{n-1} \T_0 \T_{n-1}$ and likewise for $\F_0$. 
\end{Lemma}
\begin{proof}
Note that $\phi_1 \dots \phi_{n-1} \cong \phi_0^{-1}$, we have $\E_{1,1} \cong \phi_0 \E_1 \phi_0^{-1}$ and, from (\ref{eq:T0def}),
$$\T_0 \cong \phi_0^{-1} \T_1^{-1} \dots \T_{n-1}^{-1} \dots \T_1^{-1} \cong \phi_0^{-1} \T_{n-1}^{-1} \dots \T_1^{-1} \dots \T_{n-1}^{-1}.$$
So we get
\begin{align*}
\E_0
&\cong \T_1^{-1} (\T_{n-1} \dots \T_1 \dots \T_{n-1}) \phi_0 \E_1 \phi_0^{-1} (\T_{n-1}^{-1} \dots \T_1^{-1} \dots \T_{n-1}^{-1}) \T_1 \\
&\cong \T_{n-1} \dots \T_3 \T_1^{-1} \T_2 \T_1 \dots \T_{n-1} \E_{1,1} \T_{n-1}^{-1} \dots \T_1^{-1} \T_2^{-1} \T_1 \T_3^{-1} \dots \T_{n-1}^{-1} \\
&\cong \T_{n-1} \dots \T_2 \T_1 \E_{1,1} \T_1^{-1} \T_2^{-1} \dots \T_{n-1}^{-1}
\end{align*}
where, to obtain the last isomorphism, we used that $\T_i$ commutes with $\E_{1,1}$ if $i > 2$. 

A similar argument shows that 
$$\T_{n-1}^{-1} \T_0^{-1} \E_{n-1} \T_0 \T_{n-1} \cong \T_1 \dots \T_{n-1} \E_{n-1,1} \T_{n-1}^{-1} \dots \T_1^{-1}.$$
To finish the proof note that 
$$(\T_1 \dots \T_{n-1})^{-1}(\T_{n-1} \dots \T_1) \cong (\T_{n-2} \T_{n-1}^{-1}) \dots (\T_2 \T_3^{-1}) (\T_1 \T_2^{-1}).$$
Applying $\T_i \T_{i+1}^{-1} \E_{i,1} \T_{i+1} \T_i^{-1} \cong \E_{i+1,1}$ repeatedly we get that 
$$(\T_1 \dots \T_{n-1})^{-1}(\T_{n-1} \dots \T_1) \E_{1,1} (\T_{n-1} \dots \T_1)^{-1} (\T_1 \dots \T_{n-1}) \cong \E_{n-1,1}.$$
This completes the proof. 
\end{proof}

\begin{Corollary}\label{cor:app1}
We have 
$$ (\E_0 \1_\uk)^R \cong \1_\uk \F_0 \la \la \uk, \alpha_0 \ra + 1 \ra \ \ \text{ and } \ \ (\E_0 \1_\uk)^L \cong \1_\uk \F_0 \la - \la \uk, \alpha_0 \ra - 1 \ra$$ 
\begin{align*}
\E_0 \F_0 \1_\uk &\cong \F_0 \E_0 \1_\uk \bigoplus_{[\la \uk,\alpha_0 \ra]} \1_\uk \ \ \text{ if } \la \uk, \alpha_0 \ra \ge 0 \\
\F_0 \E_0 \1_\uk &\cong \E_0 \F_0 \1_\uk \bigoplus_{[-\la \uk,\alpha_0 \ra]} \1_\uk \ \ \text{ if } \la \uk, \alpha_0 \ra \le 0
\end{align*}
where $\alpha_0 = - \sum_{i=1}^{n-1} \alpha_i$. Moreover, $\F_0 \E_i \cong \E_i \F_0$ and $\E_0 \F_i \cong \F_i \E_0$ if $i \ne 0$. 
\end{Corollary}
\begin{proof}
We have
\begin{align*}
\E_0 \F_0 \1_\uk 
&\cong \T_1 \T_0 \E_1 \F_1 \1_{s_0 s_1 \cdot \uk} \T_0^{-1} \T_1^{-1} \\
&\cong \T_1 \T_0 \left( \F_1 \E_1 \1_{s_0 s_1 \cdot \uk} \bigoplus_{[\la \uk, \alpha_0 \ra]} \1_{s_0 s_1 \cdot \uk} \right) \T_0^{-1} \T_1^{-1} \\
&\cong \F_0 \E_0 \1_\uk \bigoplus_{[\la \uk, \alpha_0 \ra]} \1_{\uk}
\end{align*}
where we used that $\la s_0 s_1 \cdot \uk, \alpha_1 \ra = \la \uk, s_1 s_0 \cdot \alpha_1 \ra = \la \uk, \alpha_0 \ra$ in the second line. The other commutation relation as well as the adjunction properties follow similarly. 

Finally, we need to check that  $\E_0$ commutes with $\F_i$ if $i \ne 0$. First note that if $i=2,\dots,n-2$ then $\F_i$ commutes with $\phi_0$ and also with $\T_1 \dots \T_{n-1} \dots \T_1$. This implies that it commutes with $\T_0$ and then 
\begin{align*}
\F_i \E_0 &\cong \F_i \T_1^{-1} \T_0^{-1} \E_1 \T_0 \T_1 \cong \T_1^{-1} \T_0^{-1} \F_i \E_1 \T_0 \T_1 \\
&\cong \T_1^{-1} \T_0^{-1} \E_1 \F_i \T_0 \T_1 \cong \T_1^{-1} \T_0^{-1} \E_1 \T_0 \T_1 \F_i \cong \E_0 \F_i
\end{align*}
where we also assumed that $i \ne 2$ to get that $\F_i$ commutes with $\T_1$. If $i=2$ then we use Lemma \ref{lem:app1} to write $\E_0$ as $\T_{n-1}^{-1} \T_0^{-1} \E_{n-1} \T_0 \T_{n-1}$ and repeat the argument. 

Finally, using the expression for $\E_0$ from the proof of Lemma \ref{lem:app1}, we have
\begin{align*}
\F_1 \E_0 
&\cong \F_1 \T_{n-1} \dots \T_1 \E_{1,1} \T_1^{-1} \dots \T_{n-1}^{-1} \\
&\cong \T_{n-1} \dots \T_3 \F_1 \T_2 \T_1 \E_{1,1} \T_1^{-1} \dots \T_{n-1}^{-1} \\
&\cong \T_{n-1} \dots \T_1 \F_2 \E_{1,1} \T_1^{-1} \dots \T_{n-1}^{-1} \\
&\cong \T_{n-1} \dots \T_1 \E_{1,1} \F_2 \T_1^{-1} \dots \T_{n-1}^{-1} \cong \E_0 \F_1 
\end{align*}
where to get the second last isomorphism we used that $\E_{1,1}$ commutes with $\F_i$ if $i \ne 1$. The proof that $\F_{n-1} \E_0 \cong \E_0 \F_{n-1}$ is similar. This completes the proof that $\F_i$ and $\E_0$ commute whenever $i \ne 0$. Taking adjoints also shows that $\E_i$ and $\F_0$ commute. 
\end{proof}

\begin{Remark}\label{rem:primes}
\cite{CK4} also discussed a ``shifted'' action using generators $\E'_{i,1}$ and $\F'_{i,-1}$ in place of $\E_{i,1}$ and $\F_{i,-1}$. If we use these generators (together with $\E_i,\F_i$ for $i \in I$) then the complexes $\T'_i$ and $\T'_{i,1}$ (defined as in Section \ref{sec:our_actions}) together with $\phi'_i \coloneqq \T'_i \T'_{i,1}$ can be used to define $\E_0$ and $\F_0$ as above. Then these $\E_0,\F_0$ also satisfy the relations in Corollary \ref{cor:app1}.
\end{Remark}

\section{Some deformation theory}\label{sec:appdeform}

We recall some results regarding deformations and obstructions classes. We use \cite{HT} as a reference. 

Consider a flat deformation $\pi\colon \tY \rightarrow \bA^n$ of a smooth variety $Y = \pi^{-1}(0)$. We suppose $\tY$ carries a $\C^\times$ action and that $\pi$ is equivariant if we equip $\bA^n$ with the dilating action $t \cdot \ux = t^2 \ux$ where $\ux = (x_1, \dots, x_n) \in \bA^n$. Let $i\colon Y \rightarrow \tY$ denote the inclusion. 

Let $h = id \times i\colon Y \times Y \rightarrow Y \times \tY$ and consider $H \coloneqq h^* h_* \O_{\Delta_Y}$ where $\Delta_Y \subset Y \times Y$ is the diagonal. A standard calculation shows that $\sH^0(H) \cong \O_{\Delta_Y}$ and $\sH^{-1}(H) \cong \O_Y \otimes_\C \Omega(\bA^n)_0$ where $\Omega(\bA^n)_0 \cong \C^n \{-2\}$ is the fiber of the cotangent bundle at $0 \in \bA^n$. 

This allows us to define
$$\varpi(Y/\tY) \in \Ext^2_{Y \times Y}(\O_{\Delta_Y}, \O_{\Delta_Y} \otimes_\C \Omega(\bA^n)_0)$$ 
as the extension class of the exact triangle
$$\sH^{-1}(H)[1] \rightarrow \tau^{\ge -1}(H) \rightarrow \sH^0(H)$$
(cf. \cite[Def. 2.9]{HT}). In particular, it gives a linear map 
$$\varpi\colon T_0(\bA^n) \rightarrow \Hom_{Y \times Y}(\O_{\Delta_Y}, \O_{\Delta_Y} [2] \{-2\})$$
where $T_0(\bA^n)$ denotes the tangent space at $0 \in \bA^n$. 

Now, for a tangent vector $v \in T_0(\bA^n)$, denote by $\pi_v\colon \tY_v \rightarrow \Spec (\C[\vareps]/\vareps^2)$ the pullback of $\pi\colon \tY \rightarrow \bA^n$ via the map $\Spec (\C[\vareps]/\vareps^2) \rightarrow \bA^n$ corresponding to $v$. Consider an object $\sM \in D(Y)$. One of the main result of \cite{HT} (Corollary 3.4) states that $\sM$ deforms to $\tY_v$ if and only if the map 
$$\O_{\Delta_Y} * \sM \xrightarrow{\varpi(v) * id} \O_{\Delta_Y} * \sM [2]\{-2\}$$
is zero, where $*$ denotes convolution with $\sM$ thought of as a kernel on $Y \times pt$.

More generally, if one has two deformations $\tY_1 \rightarrow \C^n$ and $\tY_2 \rightarrow \C^n$ and two vectors $v_1,v_2 \in T_0(\bA^n)$ then we can consider the restriction $(\tY_1 \times \tY_2)_{(v_1,v_2)} \rightarrow \Spec (\C[\vareps]/\vareps^2)$ where $(v_1,v_2)$ is a tangent vector over the origin in $\bA^n \times \bA^n$. Then a sheaf (kernel) $\sM \in D(Y_1 \times Y_2)$ deforms to $(\tY_1 \times \tY_2)_{(v_1,v_2)}$ if and only if 
$$\O_{\Delta_{Y_2}} * \sM * \O_{\Delta_{Y_1}} \xrightarrow{- \varpi(v_2) * id * id + id * id * \varpi(v_1)} \O_{\Delta_{Y_2}} * \sM * \O_{\Delta_{Y_1}} [2]\{-2\}$$
is zero.

\mbox{}

\end{document}